\documentclass[pdftex]{article}

\usepackage{amsmath,amssymb,amsthm,tikz,geometry}

\title{On Guaspari's problem about partially conservative sentences}

\author{Taishi Kurahashi\thanks{Graduate School of System Informatics, Kobe University, Japan} \thanks{\texttt{kurahashi@people.kobe-u.ac.jp}}
\and
Yuya Okawa\thanks{Graduate School of Science and Engineering, Chiba University, Japan} \thanks{\texttt{ahga4770@chiba-u.jp}}
\and
V. Yu. Shavrukov\thanks{\texttt{v.yu.shavrukov@gmail.com}}
\and
Albert Visser\thanks{Department of Philosophy, Utrecht University, Netherlands} \thanks{\texttt{a.visser@uu.nl}}}

\date{}

\theoremstyle{plain}
\newtheorem{thm}{Theorem}[section]
\newtheorem{lem}[thm]{Lemma}
\newtheorem{prop}[thm]{Proposition}
\newtheorem{cor}[thm]{Corollary}
\newtheorem{fact}[thm]{Fact}
\newtheorem{cl}{Claim}

\theoremstyle{definition}
\newtheorem{defn}[thm]{Definition}

\newtheorem{rem}[thm]{Remark}
\newtheorem{prob}[thm]{Problem}
\newtheorem{conv}[thm]{Convention}

\newcommand{\PA}{\mathsf{PA}}
\newcommand{\Prv}{\mathrm{Pr}}
\let\Pr\Prv
\newcommand{\Prf}{\mathrm{Prf}}
\newcommand{\Con}{\mathrm{Con}}
\newcommand{\Cons}{\mathrm{Cons}}
\newcommand{\HCons}{\mathrm{HCons}}
\newcommand{\RFN}{\mathrm{RFN}}
\newcommand{\Th}{\mathrm{Th}}
\newcommand{\ThP}{\mathrm{Th}_{\Pi_{n}}}
\newcommand{\ThS}{\mathrm{Th}_{\Sigma_{n}}}
\newcommand{\ThG}{\mathrm{Th}_{\Gamma}}
\newcommand{\ThD}{\mathrm{Th}_{\Gamma^{d}}}
\newcommand{\gn}[1]{\ulcorner#1\urcorner}

\def\bseq#1{{\boldsymbol#1}}

\def\J{{\mathrm J}}
\def\Ik{{\mathrm I}_{k}}
\def\pn{\Pi_{n}}
\def\pnmo{\Pi_{n-1}}
\def\sn{\Sigma_{n}}
\let\dott\dot\def\dot#1{\mathbin{\dott#1}}

\begin{document}

\maketitle

\begin{abstract}
We investigate sentences which are simultaneously partially conservative over several theories. 
First, we generalize Bennet's results on this topic to the case of more than two theories. 
In particular,  for any finite family $\{T_i\}_{i \leq k}$ of consistent r.e.~extensions of Peano Arithmetic, we give a necessary and sufficient condition for the existence of a $\Pi_n$ sentence which is unprovable in $T_i$ and $\Sigma_n$-conservative over $T_i$ for all $i \leq k$. 
Secondly, we prove that for any  finite family of such theories, there exists a $\Sigma_n$ sentence which is simultaneously unprovable and $\Pi_n$-conservative over each of these theories. 
This constitutes a positive solution to a particular case of Guaspari's problem. 
Finally, we demonstrate several non-implications among related properties of families of theories. 
\end{abstract}

\section{Introduction}

Let $T$ be a recursively enumerable (r.e.)\ consistent extension of Peano Arithmetic $\PA$.
Let $\Gamma$ denote either $\Sigma_{n}$ or $\Pi_{n}$ for some $n \geq 1$. 
Also $\Th(T)$ denotes the set of all sentences provable in $T$ and $\ThG(T)$ denotes the set of all $\Gamma$ sentences provable in $T$. 
We say a sentence $\varphi$ is \textit{$\Gamma$-conservative over $T$} if for any $\Gamma$ sentence $\psi$, $T \vdash \psi$ whenever $T+\varphi \vdash \psi$. 
Define $\Cons(\Gamma, T)$ to be the set of all $\Gamma^{d}$ sentences which are $\Gamma$-conservative over $T$ where $\Sigma_{n}^{d} = \Pi_{n}$ and $\Pi_{n}^{d} = \Sigma_{n}$. 
Guaspari \cite{Gua} proved that $\Cons(\Gamma, T) \setminus \Th(T)$ is non-empty, that is, there exist $\Gamma^d$ sentences which are $\Gamma$-conservative over $T$ and unprovable in $T$. 
 Guaspari also asked the following question (\cite[Question~5(1)]{Gua}).

\begin{quotation}
{\noindent\small
\ldots~if $\langle T_{i} \mid i \in \omega \rangle$ is an r.e.~sequence of r.e.~theories is there a $\Gamma$ sentence which is independent and $\breve{\Gamma}$-conservative\footnote{Guaspari's $\breve{\Gamma}$ is our $\Gamma^d$.} over each $T_{i}$?
The question is open even for sequences of length 2.
}
\end{quotation}

Guaspari actually proved that for any theory $T$, there are $\Gamma^{d}$ sentences which are $T$-unprovable and simultaneously $\Gamma$-conservative over all subtheories of $T$. 
Thus for the family of all subtheories of $T$, Guaspari's question has an affirmative answer.  

On the other hand, Misercque~\cite{Mis,MisPhD} supplied a negative answer to the general version of Guaspari's problem. 
That is, Misercque found an infinite r.e.~family $\{T_i\}_{i \in \omega}$ of theories such that there is no $\Gamma^d$ sentence which is simultaneously unprovable and $\Gamma$-conservative over $T_i$ for all $i \in \omega$. 
The infinity of the family was essential to Misercque's argument for general~$\Gamma$.
Misercque~\cite{Mis,MisPhD} also presented an
example of consistent theories $T_0$ and~$T_1$ such that $T_{1}$ proves every $\pn$~sentence which is $\sn$-conservative over~$T_{0}$.
Thus the version of Guaspari's problem with two theories and $\Gamma = \Sigma_{n}$ is also settled negatively.

Bennet \cite{Ben86, Ben} also investigated Guaspari's problem for two theories. 
Bennet firstly proved that the statement $\bigcap_{i \leq 1}\Cons(\Gamma, T_{i}) \setminus \Th(T_{i})\neq \emptyset$ saying that ``there exists a $\Gamma^d$ sentence which is simultaneously unprovable and $\Gamma$-conservative over $T_0$ and $T_1$'' is equivalent to ``$\Cons(\Gamma, T_{0}) \setminus \Th(T_{1}) \neq \emptyset$ and $\Cons(\Gamma, T_{1}) \setminus \Th(T_{0}) \neq \emptyset$''. 
Thus the investigation of Guaspari's problem for two theories is reduced to that of simultaneous conditions of the form $\Cons(\Gamma, T) \setminus \Th(U) \neq \emptyset$. 
On~top of that, Bennet showed that the condition ``$\ThD(T) \nsubseteq \Th(U)$ or $U + \ThG(T)$ is consistent'' is sufficient for $\Cons(\Gamma, T) \setminus \Th(U) \neq \emptyset$. 
Furthermore, he proved that in the case of $\Gamma = \Sigma_{n}$, ``$\ThP(T) \nsubseteq \Th(U)$ or $U + \ThS(T)$ is consistent'' is, in~fact, equivalent to $\Cons(\Sigma_{n}, T) \setminus \Th(U) \neq \emptyset$. 

In~the case of $\Gamma = \Pi_{n}$, Bennet established that $\Cons(\Pi_{n}, T) \setminus \Th(U) \neq \emptyset$ generally 
fails to imply ``$\ThS(T) \nsubseteq \Th(U)$ or $U + \ThP(T)$ is consistent'', parting with $\Pi/\Sigma$ symmetry. 
Guaspari's problem for finitely many theories and $\Gamma = \Pi_{n}$ has up till now remained open. 

Against this background, we proceed with an investigation of Guaspari's problem in the case of three or more theories, based on Bennet's approach.
In particular, in the present paper, we completely solve Guaspari's problem for finitely many theories: for any finite family $\{T_i\}_{i \leq k}$ of consistent theories, 
\begin{enumerate}
	\item we give a necessary and sufficient condition for the non-emptiness of the set $\bigcap_{i \leq k}\Cons(\Sigma_n, T_{i}) \setminus \Th(T_{i})$; and
	\item we prove that $\bigcap_{i \leq k}\Cons(\Pi_n, T_{i}) \setminus \Th(T_{i})$ is never empty. 
\end{enumerate}
The latter contribution contrasts with the earlier negative solutions to Guaspari's problem in the $\Gamma=\Sigma_{n}$ and the infinitary cases.
We also briefly investigate Guaspari's problem for infinite r.e.~families of theories. 

In Section \ref{Back}, we survey already known results concerning Guaspari's problem. 
In Section \ref{Pre}, we introduce some notation and facts. 
In Section \ref{Gen}, we generalize Bennet's results referred~to above to the case of  more than two theories. 
Among other things, we prove that for any r.e.~family $\{T_i\}_{i \in \omega}$ of theories, if there exists an r.e.~set $X$ of natural numbers such that 
\[
\bigcap_{i \notin X} \Th_{\Gamma^d}(T_i) \nsubseteq \Th\bigg(U + \bigcup_{i \in X} \Th_{\Gamma_n}(T_i)\bigg),
\]
then $\bigcap_{i \in \omega} \Cons(\Gamma, T_{i})\setminus \Th(U)$ is non-empty. 
Moreover, we prove that the converse implication also holds for finite families of theories in the case of $\Gamma = \Sigma_n$. 
In Section \ref{Picons}, we give an affirmative answer to Guaspari's problem for finite families of theories and $\Gamma = \Pi_n$. 
In Section \ref{SecCex}, we show the failure of several implications between properties of families of theories related to Guaspari's problem.

\section{Background}\label{Back}

Throughout this paper, all theories considered are r.e.~consistent extensions of Peano Arithmetic $\PA$ in the same language, hence we call such a theory simply a \textit{theory}. 
Let $\omega$ be the set of all natural numbers. 
Throughout this paper, we assume that $n$ always denotes a non-zero natural number. 
The classes $\Sigma_{n}$ and $\Pi_{n}$ of formulas are defined as usual. 
We also assume that $\Gamma$ denotes either $\Sigma_{n}$ or $\Pi_{n}$. 
Let $\Sigma_{n}^{d} = \Pi_{n}$ and $\Pi_{n}^{d} = \Sigma_{n}$. 
We define the following sets. 

\begin{defn} Let $T$ be a theory and $M$ be a model. 
\begin{itemize}
	\item $\Th(T):=\{\,\text{sentences $\varphi$}:T\vdash\varphi\,\}$.
	\item $\ThG(T):=\{\,\text{sentences $\varphi\in\Gamma$}:T\vdash\varphi\,\}$.
	\item $\ThG(M):=\{\,\text{sentences $\varphi\in\Gamma$}:M\models\varphi\,\}$.
\end{itemize}
\end{defn}

The notion of partially conservative sentences has appeared in the context of the incompleteness theorems. 
For example, Kreisel \cite[Remark~14(ii)]{Kre} showed that the negation of the conventional consistency statement $\Con_T$ of $T$ is $\Pi_{1}$-conservative over $T$, that is, for any $\Pi_{1}$ sentence $\pi$, 
one has $T \vdash \pi$ whenever $T + \lnot \Con_{T} \vdash \pi$. 
This is an extension of G\"odel's second incompleteness theorem. 
For another example, Smory\'nski \cite[Application~5]{Smo} proved that $T$ is $\Sigma_{1}$-sound if and only if every $T$-undecidable $\Pi_{1}$ sentence is $\Sigma_{1}$-conservative over $T$. 
Also, Smory\'nski proved that $T$ is $\Sigma_{1}$-sound if and only if $\Con_{T}$ is $\Sigma_{1}$-conservative over $T$.


Guaspari investigated the general concept of $\Gamma$-conservativity in \cite{Gua} (see also H\'ajek \cite{Haj79} and Lindstr\"om \cite{Lin84}). 

\begin{defn}\leavevmode Let $T$ be any theory.
\begin{itemize}
	\item A sentence $\varphi$ is said to be \textit{$\Gamma$-conservative over $T$} if for all $\Gamma$ sentences $\psi$, if $T + \varphi \vdash \psi$, then $T \vdash \psi$.
Equivalently, $\ThG(T+\varphi)\subseteq\ThG(T)$.
	\item Let $\Cons(\Gamma, T) := \{\varphi \in \Gamma^{d} : \varphi$ is $\Gamma$-conservative over $T\}$.
\end{itemize}
\end{defn}

Deviating from the expositions in Bennet \cite{Ben86, Ben} and Lindstr\"om \cite{Lin84,Lin},
we restrict the elements of $\Cons(\Gamma, T)$ to $\Gamma^{d}$ sentences because the latter are the focus of interest for the present paper. 

Every $T$-provable $\Gamma^{d}$ sentence $\varphi$ trivially belongs to $\Cons(\Gamma, T)$. 
Guaspari proved that every theory has non-trivially $\Gamma$-conservative $\Gamma^{d}$ sentences, that is,
\begin{fact}[Guaspari {\cite[Theorem 2.4]{Gua}}]\label{exCons}
For any theory $T$, $\Cons(\Gamma, T) \setminus \Th(T) \neq \emptyset$.
\qed
\end{fact}

If $T \vdash \lnot \varphi$, then $T+ \varphi$ is inconsistent, and hence $\varphi$ is not $\Gamma$-conservative over $T$ because $T$ is consistent. 
This shows that if $\varphi \in \Cons(\Gamma, T) \setminus \Th(T)$, then $\varphi$ is undecidable in $T$. 
Therefore Fact \ref{exCons} can be thought as a strengthening of G\"odel--Rosser's first incompleteness theorem.


In this paper, many properties of uniformly r.e.~collections of theories are meaningful both for finite and for infinite collections. 
We~use the term \emph{r.e.~family of theories} $\{T_i\}_{i\in\J}$
to stand for a sequence of theories with $T_i$ being uniformly r.e.~in~$i$. 
The~index set $\J$ is a non-empty initial segment of~$\omega$, that~is, 
$\J\in\{\Ik\}_{k\in\omega}\cup\{\omega\}$, where $\Ik=\{0,\ldots,k\}$.

Mostowski proved the following generalization of G\"odel--Rosser's first incompleteness theorem.
\begin{fact}[Mostowski {\cite[Theorem 1]{Mos}}]\label{gGR}
Let $\{T_{i}\}_{i \in\J}$ be an r.e.~family of theories. 
Then there is a $\Pi_{1}$~sentence $\varphi$ such that $\varphi$, $\lnot \varphi \notin \bigcup_{i\in\J}\Th(T_{i})$.
\qed
\end{fact}

It is then natural to expect the existence of a sentence which is simultaneously $\Gamma$-conservative over several theories. 
Guaspari proposed the following problem:

\begin{prob}[Guaspari {\cite[Question~5(1)]{Gua}}]\label{Mainprob}
Given an r.e.~family $\{T_{i}\}_{i \in\J}$ of theories, must $\bigcap_{i\in\J}\Cons(\Gamma, T_{i}) \setminus \Th(T_{i}) \neq \emptyset$ hold?
\end{prob}

 Guaspari pointed out that this problem is open even for pairs of theories. 
In the remainder of this subsection, we survey known results concerning Guaspari's problem.

Guaspari actually proved a stronger result than Fact \ref{exCons}, to wit that there exists a $\Gamma^d$ sentence which is simultaneously $\Gamma$-conservative over all sufficiently strong subtheories of $T$. 
Such sentences are called hereditarily $\Gamma$-conservative. 

\begin{defn} Let $T$ be any theory.
\begin{itemize}
	\item A sentence $\varphi$ is said to be \textit{hereditarily $\Gamma$-conservative over $T$} if for all theories $S$ such that $T \vdash S \vdash \PA$, $\varphi$ is $\Gamma$-conservative over $S$.
	\item Let $\HCons(\Gamma, T) := \{\varphi \in \Gamma^{d} : \varphi$ is hereditarily $\Gamma$-conservative over $T\}$.
\end{itemize}
\end{defn}

We also restrict the elements of $\HCons(\Gamma, T)$ to $\Gamma^{d}$ sentences as in the case of $\Cons(\Gamma, T)$.

\begin{fact}[Guaspari {\cite[Theorem 2.6]{Gua}}]\label{exHCons} 
For any theory $T$, $\HCons(\Gamma, T) \setminus \Th(T) \neq \emptyset$.
\qed
\end{fact}

Fact \ref{exHCons} is strengthened by Lindstr\"om as follows. 
We say a set $X$ of sentences is \textit{pointwise consistent with a theory $T$} if $T + \varphi$ is consistent for each $\varphi \in X$. 

\begin{fact}[Lindstr\"om {\cite[Corollary 1]{Lin84}}]\label{Lin84}
Let\/ $T$ be a theory and\/ $X$ an r.e.~set of sentences which is pointwise consistent with $T$. 
Then $\HCons(\Gamma, T) \setminus X \neq \emptyset$. 
\qed
\end{fact}

Following Guaspari's study, Misercque and Bennet also investigated Guaspari's Problem \ref{Mainprob}. 
Misercque proved that Guaspari's problem does not generally admit a positive solution.

\begin{fact}[Misercque {\cite[Theorem 2.1]{Mis} or \cite[Proposition~5.1.3]{MisPhD}}]\label{MsCex}
There is an infinite r.e.~family $\{T_{i}\}_{i \in \omega}$ of theories such that 
for all $\Gamma$, $\bigcap_{i\in\omega}\Cons(\Gamma, T_{i}) \setminus \Th(T_{i}) = \emptyset$.
\qed
\end{fact}

Since Misercque's family of theories is not finite, it is natural to see Guaspari's problem restricted to finite families of theories as a separate challenge. 
Misercque and Bennet analyzed the existence of $\Gamma^d$ sentences which are simultaneously $\Gamma$-conservative over two theories. 
Bennet showed that Guaspari's problem for two theories can be reduced to a more easily studied problem. 

\begin{fact}[Bennet {\cite[Corollary~8]{Ben86}} or {\cite[Corollary 3.1.9]{Ben}}]\label{B1}
For any theories $T_{0}$ and $T_{1}$, the following are equivalent:
\begin{enumerate}
	\item $\bigcap_{i \leq 1}\Cons(\Gamma, T_{i}) \setminus \Th(T_{i})\neq \emptyset$;
	\item $\Cons(\Gamma, T_{0}) \setminus \Th(T_{1}) \neq \emptyset$ and $\Cons(\Gamma, T_{1}) \setminus \Th(T_{0}) \neq \emptyset$.\qed
\end{enumerate}
\end{fact}

Therefore, the investigation of Guaspari's problem for two theories is equivalent to that of conditions of the form $\Cons(\Gamma, T) \setminus \Th(U) \neq \emptyset$.
Bennet found a sufficient condition for $\Cons(\Gamma, T) \setminus \Th(U) \neq \emptyset$:

\begin{fact}[Bennet {\cite[p.\,67]{Ben86}} or {\cite[p.\,38]{Ben}}; see also
   Misercque~{\cite[Proposition~5.2.3]{MisPhD}}]\label{B2'}
Let\/ $T$ and\/ $U$ be theories. 
Suppose $\ThD(T) \not \subseteq \Th(U)$ or $U + \ThG(T)$ is consistent.
Then $\Cons(\Gamma, T) \setminus \Th(U) \neq \emptyset$.
\qed
\end{fact}

In the case of $\Gamma = \Sigma_{n}$, this sufficient condition is also necessary.

\begin{fact}[Bennet {\cite[Theorem 6]{Ben86}} or {\cite[Theorem 3.1.7]{Ben}}]\label{B3'}
For any theories $T$ and $U$, the following are equivalent:
\begin{enumerate}
	\item $\Cons(\Sigma_{n}, T) \setminus \Th(U) \neq \emptyset$;
	\item $\ThP(T) \not \subseteq \Th(U)$ or $U + \ThS(T)$ is consistent.\qed
\end{enumerate}
\end{fact}

Let $\varphi$ be a sentence such that $\varphi \in \Cons(\Pi_{n}, \PA) \setminus \Th(\PA)$ (see Fact \ref{exCons}). Let $T_0 := \PA + \varphi$ and $T_1 := \PA + \lnot \varphi$.
Then, it is easy to see $\ThP(T_0) \subseteq \Th(T_1)$ and $T_1 + \ThS(T_0)$ is inconsistent. 
Hence, $\Cons(\Sigma_{n}, T_0) \setminus \Th(T_1) = \emptyset$ by Fact \ref{B3'}.
Therefore, by Fact \ref{B1}, $\bigcap_{i\leq1}\Cons(\Sigma_{n},T_{i}) \setminus \Th(T_{i})= \emptyset$. 
Thus Guaspari's problem is answered negatively for the pair $T_{0}$, $T_{1}$ and $\Gamma = \Sigma_{n}$ (see Misercque \cite[Theorem~2.2]{Mis} or~\cite[Proposition~5.1.2]{MisPhD}
or Lindstr\"om \cite[Exercise 5.9(a)]{Lin}). 

Bennet proved that the condition $\Cons(\Pi_{n}, T) \setminus \Th(U) \neq \emptyset$ cannot be characterized as in Fact \ref{B3'}. 

\begin{fact}[Bennet {\cite[pp.\,67--68]{Ben86}} or {\cite[Corollary 3.2.6]{Ben}}]\label{CexPi}
There are $T$ and $U$ satisfying the following conditions:
\begin{enumerate}
	\item $\Cons(\Pi_{n}, T) \setminus \Th(U) \neq \emptyset$;
	\item $\ThS(T) \subseteq \Th(U)$;
	\item $U + \ThP(T)$ is inconsistent.\qed
\end{enumerate}
\end{fact}

For two theories, the major remaining case of Guaspari's problem is the case of $\Gamma = \Pi_{n}$:

\begin{prob}[Misercque {\cite[Probl\`eme~7]{MisPhD}}]\label{MisProb}
Are there $n\geq1$ and theories $T_{0}$ and~$T_{1}$ such that 
\[
\bigcap_{i \leq 1}\Cons(\pn, T_{i}) \setminus \Th(T_{i})=\emptyset\;?
\]
\end{prob}

Bennet's analysis relates this problem to

\begin{prob}[Bennet {\cite[Q3]{Ben86}} or {\cite[p.\,38]{Ben}}]\label{BenProb}
Are there theories $T$ and $U$ such that 
\[
\Cons(\Pi_{n}, T) \setminus \Th(U) = \emptyset\;?
\]
\end{prob}

In Section~\ref{Picons}, we
shall obtain a negative answer to Bennet's Problem~{\ref{BenProb}} which will
enable us to settle Problem~{\ref{MisProb}} for all finite families of theories.

Bennet also investigated a variant of Guaspari's problem for hereditarily $\Gamma$-conservative sentences. 
He proved the following equivalence concerning the condition $\bigcap_{i \leq 1}\HCons(\Gamma, T_{i}) \setminus \Th(T_{i})\neq \emptyset$, which corresponds to Fact~\ref{B1}. 


\begin{fact}[Bennet {\cite[Theorem~4 and Corollary~5]{Ben86}} or {\cite[Theorem~3.1.5 and Corollary~3.1.6]{Ben}}]\label{HCfB1}
For any theories $T_{0}$ and $T_{1}$, the following are equivalent:
\begin{enumerate}
	\item $\bigcap_{i \leq 1}\HCons(\Gamma, T_{i}) \setminus \Th(T_{i}) \neq \emptyset$;
	\item $\HCons(\Gamma, T_{0}) \setminus \Th(T_{1}) \neq \emptyset$ and $\HCons(\Gamma, T_{1}) \setminus \Th(T_{0}) \neq \emptyset$.\qed
\end{enumerate}
\end{fact}

Bennet characterized the condition $\HCons(\Gamma, T) \setminus \Th(U) \neq \emptyset$ by 
employing the method Misercque used in his proof of Fact \ref{MsCex} and the following lemma by Guaspari. 
Let $\overline{m}$ denote the numeral for a natural number~$m$. 

\begin{fact}[Guaspari {\cite[Lemma 2.10]{Gua}}]\label{F1}
For any r.e.~set $X \subseteq \omega$, there exists a $\Gamma$ formula $\delta(x)$ satisfying the following conditions for any $m \in \omega$: 
\begin{enumerate}
	\item If\/ $m \in X$, then $T \vdash \delta(\overline{m})$; 
	\item If\/ $m \notin X$, then $\neg \delta(\overline{m}) \in \HCons(\Gamma, T)$.\qed
\end{enumerate}
\end{fact}

In this paper, we generalize Bennet's results without formally relying on said results except for the following fact. Fact \ref{HCfB2} will be used to establish a generalization of itself. We therefore include Bennet's proof.
Fix a natural g\"odelnumbering, and for any 
formula $\varphi$, let $\gn{\varphi}$ denote the numeral for the g\"odelnumber of $\varphi$. 

\begin{fact}[Bennet {\cite[Theorem 4]{Ben86}} or {\cite[Theorem 3.1.5]{Ben}}]\label{HCfB2}
For any theories $T$ and $U$, the following are equivalent:
\begin{enumerate}
	\item $\HCons(\Gamma, T) \setminus \Th(U) \neq \emptyset$;
	\item $U + \ThG(T)$ is consistent.
\end{enumerate}
\end{fact}

\begin{proof}
$1 \Rightarrow 2$: Suppose $U + \ThG(T)$ is inconsistent. 
Then, there exists a $\Gamma$~sentence $\psi$ such that $T \vdash \psi$ and $U \vdash \lnot \psi$. Let $\varphi \in \HCons(\Gamma, T)$. 
Since $\PA + (\psi \lor \lnot \varphi)$ is a subtheory of $T$ 
and $\PA + (\psi \lor \lnot \varphi) + \varphi \vdash \psi$, we obtain $\PA + (\psi \lor \lnot \varphi) \vdash \psi$. 
Then, $\PA + \lnot \psi \vdash \varphi$ and hence, $U \vdash \varphi$. 

$2 \Rightarrow 1$: Suppose $U + \ThG(T)$ is consistent. 
Fact \ref{F1}, when applied to $X = \{\gn{\varphi} : \varphi \in \Th(U)\}$, yields a $\Gamma$ formula $\delta(x)$ such that for any sentence $\varphi$, 
\begin{itemize}
	\item[(a)] If $U \vdash \varphi$, then $T \vdash \delta(\gn{\varphi})$; 
	\item[(b)] If $U \nvdash \varphi$, then $\lnot \delta(\gn{\varphi}) \in \HCons(\Gamma, T)$. 
\end{itemize}
Let $\psi$ be a $\Gamma^{d}$ sentence such that $\PA \vdash \psi \leftrightarrow \lnot \delta(\gn{\psi})$. 
We show $U \nvdash \psi$ and $\psi \in \HCons(\Gamma, T)$. 
Towards contradiction, assume $U \vdash \psi$. 
By (a), we have $T \vdash \delta(\gn{\psi})$, whereas $U \vdash \lnot \delta(\gn{\psi})$ by the definition of $\psi$. 
Since $\delta(\gn{\psi})$ is a~$\Gamma$ sentence, $U + \ThG(T)$ is inconsistent.  
This contradicts our supposition. 
Therefore $U \nvdash \psi$ and hence, $\psi \in \HCons(\Gamma, T)$ by (b). 
\end{proof}

Let $\varphi$ be a $\Gamma$ sentence such that $\PA \nvdash \varphi$ and $\PA \nvdash \lnot \varphi$. 
Let $T_0 := \PA + \varphi$ and $T_1 := \PA + \lnot \varphi$.
Then, $T_1 + \Th_{\Gamma}(T_0)$ is inconsistent. 
Hence, $\HCons(\Gamma, T_0) \setminus \Th(T_1) = \emptyset$ by Fact \ref{HCfB2}.
Therefore, $\bigcap_{i \leq 1}\HCons(\Gamma, T_{i}) \setminus \Th(T_{i}) = \emptyset$ by Fact~\ref{HCfB1}. 
Thus we have a negative answer to the hereditary variant of Guaspari's problem for a pair of theories.

As a corollary to Facts \ref{HCfB1} and \ref{HCfB2}, we have: 

\begin{cor}[Bennet {\cite[Corollary 5]{Ben86}} or {\cite[Corollary 3.1.6]{Ben}}]\label{HCfB3}
For any theories $T_{0}$ and $T_{1}$, the following are equivalent:
\begin{enumerate}
	\item $\bigcap_{i \leq 1}\HCons(\Gamma, T_{i})\setminus\Th(T_{i})\neq \emptyset$;
	\item $T_1 + \ThG(T_0)$ and $T_0 + \ThG(T_1)$ are consistent.\qed
\end{enumerate}
\end{cor}

\section{Preliminaries}\label{Pre}

In this section, we review some basic definitions and facts. 

Apart from the formula classes $\sn$ and~$\pn$ with $n>0$,
we also recall the class $\Delta_{0}=\Sigma_{0}=\Pi_{0}$ with its usual
definition (H\'ajek and Pudl\'ak \cite[0.30]{HP}).
We~say that a formula $\varphi$ is $\Delta_1$ if it is $\Sigma_1$ and is provably equivalent to some $\Pi_1$~formula in~$\PA$.
Seeing as $\PA$ proves collection for each formula in its language, 
we~are going to freely use the fact that, up to $\PA$-provable equivalence,
all the formula classes mentioned above are closed under bounded quantification.

We can naturally describe a formula $\Prf(X, x, y)$ saying that ``a~formula with the g\"odelnumber $x$ has a proof with the g\"odelnumber $y$ from the set $X$ of assumptions'', where $X$ is an auxiliary second-order variable.  
For each formula $\sigma(v)$, let $\Prf_\sigma(x, y)$ be the formula obtained by replacing the subformula $v \in X$ from $\Prf(X, x, y)$ with $\sigma(v)$. 
Then, a \textit{standard proof predicate} for a theory $T$ is a formula of the form $\Prf_\sigma(x, y)$, where $\sigma(v)$ is a $\Delta_1$ formula defining a set of axioms for $T$ in the standard model of arithmetic. 
Let $\Prf_T(x, y)$ denote some standard proof predicate for $T$. 
The formula $\Prv_{T}(x) :\equiv \exists y\,\Prf_{T}(x, y)$ is called a \textit{standard provability predicate} for $T$. 
The formulas $\Prf_T(x, y)$ and $\Pr_T(x)$ are $\Delta_1$ and $\Sigma_1$, respectively. 
Then, it follows from the $\Sigma_1$-soundness and $\Sigma_1$-completeness of $\PA$ that for any formula $\varphi$, $T \vdash \varphi$ if and only if $\PA \vdash \Prv_T(\gn{\varphi})$. 
Our setup is essentially identical to the one in Lindstr\"om~\cite[pp.\,15--16]{Lin}.

We introduce the witness comparison notation (cf.~Guaspari and Solovay~\cite{GS}). 

\begin{defn}
For any formulas $\varphi \equiv \exists x\,\alpha(x)$ and $\psi \equiv \exists y\,\beta(y)$, 
\begin{itemize}
	\item $\varphi \preceq \psi \equiv \exists x \, (\alpha(x) \land \forall y < x\,\lnot \beta(y))$;
	\item $\varphi \prec \psi \equiv \exists x \, (\alpha(x) \land \forall y \leq x\,\lnot \beta(y))$.
\end{itemize}
\end{defn}

\begin{fact}[cf.~Lindstr\"om {\cite[Lemma 1.3]{Lin}}]\label{dual}
For any formulas $\varphi \equiv \exists x\,\alpha(x)$ and $\psi \equiv \exists y\,\beta(y)$, 
\begin{enumerate}
    \item $\PA\vdash\varphi\preceq\psi\to\varphi$;
	\item $\PA \vdash \lnot \left(\varphi \preceq \psi \land \psi \prec \varphi \right)$;
	\item $\PA \vdash (\varphi \lor \psi) \to (\varphi \preceq \psi \lor \psi \prec \varphi)$;
	\item $\PA \vdash \varphi \land \lnot \psi \to \varphi \prec \psi $.\qed
\end{enumerate}
\end{fact}

\begin{defn}\label{rPrDef}
Let $\Gamma(x)$ be a $\Delta_1$ formula naturally expressing that ``$x$~is the g\"odelnumber of a $\Gamma$ formula'', and let ${\rm True}_{\Gamma}(x)$ be a $\Gamma$ formula saying that ``$x$~is the g\"odelnumber of a true $\Gamma$ sentence'' (see H\'ajek and Pudl\'ak \cite[I.1(d)]{HP}). 
We~define the \emph{relativized proof predicate} 
\[
\Prf_{T}^{\Gamma}(x,y)\equiv
\exists u\leq y\,\bigl(\Gamma(u)\land{\rm True}_{\Gamma}(u)\land\Prf_{T}(u\dot{\to}x,y)\bigr) 
\]
(cf.~Lindstr\"om~\cite[p.\,63]{Lin}),
where the virtual term $v\dot{\to}w$ represents the function sending the 
g\"odelnumbers of two formulas to that of the implication between them by its
natural $\Delta_{1}$~definition.
Under any reasonable g\"odelnumbering, $\Prf_{T}(u\dot{\to}x,y)$ already implies
$u\leq y$.
Note that $\Prf_{T}^{\Gamma}(x,y)$ is ($\PA$-provably equivalent~to) a
$\Gamma$~formula.

The \emph{relativized provability predicate} is
\[
\Pr_{T}^{\Gamma}(x)\equiv\exists y\,\Prf_{T}^{\Gamma}(x,y)
\]
(see 
Smory\'nski~\cite[Definition 7.3.1]{Smo85}, 
H\'ajek and Pudl\'ak~\cite[III.4.23]{HP}, or Lindstr\"om~\cite[p.\,63]{Lin}).
Observe that both $\Pr_{T}^{\sn}(x)$ and, when $n>1$, 
$\Pr_{T}^{\pnmo}(x)$ are~$\sn$.
It~can be shown in~$\PA$ that $\Pr_{T}^{\Gamma}(x)$ is equivalent to
$\Pr_{T+\textrm{True}_{\Gamma}}(x)$ with the right-to-left direction
requiring an appropriate instance of collection.

In~the interest of uniformity of exposition,
we also allow the use of the $\Delta_{1}$~formula $\Prf_{T}^{\Delta_{0}}(x,y)$ 
and the $\Sigma_{1}$~formula $\Pr_{T}^{\Delta_{0}}(x)$
defined in full analogy.
\end{defn}

\begin{fact}
\label{rPr}
Let\/ $\varphi$ be an arbitrary sentence and\/ $\gamma$ any\/ $\Gamma$~sentence.
\begin{enumerate}
	\item $T\vdash\Prf_{T}^{\Gamma}(\gn{\varphi},\overline p)\to\varphi$ 
	     for each\/ $p\in\omega$;
	\item If\/ $T+\gamma\vdash\varphi$, 
	     then\/ $\PA+\gamma\vdash\Prf_{T}^{\Gamma}(\gn{\varphi},\overline p)$
	     for some\/ $p\in\omega$;
	\item $\PA+\gamma\vdash\Pr_{T}^{\Gamma}(\gn{\gamma})$;
	\item The formulas\/ $\Pr_{T}^{\sn}(x)$ and\/ $\Pr_{T}^{\pnmo}(x)$
	     are equivalent in\/~$\PA$;
	\item As are $\Pr_{T}^{\Sigma_{1}}(x)$ and\/ $\Pr_{T}(x)$.
\end{enumerate}
\end{fact}

\begin{proof}[Comments]
1: This is a relativized form of \emph{Small Reflection Principle}~---
see e.g.\ Lindstr\"om~\cite[Lemma~5.1(ii)]{Lin}.

2: See Lindstr\"om~\cite[Lemma~5.1(iii)]{Lin}.

3 follows at once from~2.

4: Any true $\sn$~sentence $\exists x\,\pi(x)$ is a consequence of some true
$\pnmo$~sentence of the form $\pi(\overline m)$ for some natural number~$m$.
This observation is formalizable in~$\PA$.

5 is a consequence of provable $\Sigma_{1}$~completeness 
(Lindstr\"om~\cite[Fact~1.9(d)]{Lin})
\end{proof}

Even though Fact~\ref{rPr}.4 tells us that the relativized provability predicates
$\Pr_{T}^{\sn}(\cdot)$ and $\Pr_{T}^{\pnmo}(\cdot)$ are equivalent, we still have uses for both these formulas because,
when used as terms in witness comparison,
they behave differently in view of unequal quantifier complexity of the underlying
relativized proof predicates $\Prf_{T}^{\sn}(\cdot,\cdot)$ and $\Prf_{T}^{\pnmo}(\cdot,\cdot)$.
This is briefly discussed in Smory\'nski~\cite[p.\,318]{Smo85}.

%
 
It is well-known that the $\Pi_n$ sentence $\neg\Prv_{T}^{\Sigma_{n}}(\gn{0=1})$ is $\PA$-provably equivalent to the uniform $\Pi_n$~reflection principle $\RFN_T(\Pi_{n})$ for~$T$. 
The next proposition is a generalization of Kreisel's $\Pi_{1}$-conservativity result. 

\begin{conv}\label{simConv}
For each $\Gamma$ sentence $\varphi$, let ${\sim}\varphi$ denote a $\Gamma^{d}$~sentence logically equivalent to 
$\lnot\varphi$. 
\end{conv}

\begin{prop}[H\'ajek {\cite[Proposition~3]{Haj79}} for $n=2$, or Blanck~{\cite[Corollary 4.32]{Bla}}]\label{lPTS}
For any theory $T$, $\Prv_{T}^{\Sigma_{n}}(\gn{0=1})\in\Cons(\Pi_{n},T)$.
\end{prop}

\begin{proof}
Let $\pi$ be any $\Pi_{n}$ sentence such that $T + \Prv_{T}^{\Sigma_{n}}(\gn{0=1})\vdash \pi$. 
Then
\[
T\vdash\lnot\pi\to\forall u\,\bigl(\Sigma_{n}(u)\land\mathrm{True}_{\Sigma_{n}}(u)\to\lnot\Prv_{T}(u\dot{\to}\gn{0=1})\bigr).
\]
Since ${\sim}\pi$ is 
a $\Sigma_{n}$ sentence, $T\vdash \lnot \pi \to \Sigma_{n}(\gn{{\sim}\pi})\land \mathrm{True}_{\Sigma_{n}}(\gn{{\sim}\pi})$. Therefore, 
\[
T \vdash \lnot \pi \to \lnot \Prv_{T}(\gn{{\sim}\pi\to0=1}).
\]
That is, $T \vdash \lnot \pi \to \lnot \Prv_{T}(\gn{\pi})$. Hence, $T \vdash \Prv_{T}(\gn{\pi}) \to \pi$. By L\"ob's theorem, we have $T \vdash \pi$. 
\end{proof}

The following fact is used in Section~\ref{Picons}. 

\begin{fact}[See Exercise 4.2 in Lindstr\"om \cite{Lin}]\label{nCon}
For any theory~$T$, there exists a standard provability predicate $\Prv_T(x)$ for $T$ such that for all $n > 0$, $T \nvdash \Prv_{T}^{\Sigma_{n}}(\gn{0=1})$. 
\qed
\end{fact}

Finally, we prove the following useful lemma.

\begin{lem}[See Misercque~{\cite[Proposition~2.5.3]{MisPhD}} for $\Gamma=\Pi_{1}$]\label{cCons}
For any theory~$T$ and for any\/ $\Gamma^{d}$~formulas $\varphi$ and $\psi$, 
\begin{enumerate}
	\item If $\varphi$, $\psi \in \Cons(\Gamma, T)$, then $\varphi \land \psi \in \Cons(\Gamma, T)$;
	\item If $\varphi \in \Cons(\Gamma, T)$ and $T + \varphi \vdash \psi$, then $\psi \in \Cons(\Gamma, T)$.
\end{enumerate}
\end{lem}

\begin{proof}
1: Suppose $\varphi$, $\psi \in \Cons(\Gamma, T)$. 
Let $\gamma$ be a $\Gamma$ sentence such that $T + \varphi \land \psi \vdash \gamma$.
Then, $T + \varphi \vdash \psi \to \gamma$ and $\psi \to \gamma$ is a $\Gamma$ sentence. 
Since $\varphi \in \Cons(\Gamma, T)$, we have $T \vdash \psi \to \gamma$. 
Since $\psi \in \Cons(\Gamma, T)$, we obtain $T \vdash \gamma$. 

2: Suppose $\varphi \in \Cons(\Gamma, T)$ and $T + \varphi \vdash \psi$. 
Let $\gamma$ be a $\Gamma$ sentence such that $T + \psi \vdash \gamma$.
Then, we have $T + \varphi \vdash \gamma$. 
Therefore, we obtain $T \vdash \gamma$. 
\end{proof}

\section{Generalizations of Bennet's results}\label{Gen}

In this section, we extend Bennet's results discussed in Section \ref{Back} to larger families of theories. 
This section consists of two subsections. 
In Subsection \ref{GCs}, we generalize Facts \ref{B1}, \ref{B2'} and \ref{B3'} to the case of more than two theories. 
In Subsection \ref{HGCs}, we handle Facts \ref{HCfB1} and \ref{HCfB2}.

\subsection{$\Gamma$-conservative sentences}\label{GCs}

First, we generalize Fact \ref{B1} to finite families of theories. 

\begin{thm}\label{GC1}
For any\/ $k \geq 1$ and theories\/ $T_{0},\ldots,T_{k}$, the following are equivalent:
\begin{enumerate}
	\item $\bigcap_{i \leq k}\Cons(\Gamma, T_{i}) \setminus \Th(T_{i}) \neq \emptyset$;
	\item For all $i \leq k$, $\bigl(\bigcap_{\substack{j \neq i \\ j \leq k}}\Cons(\Gamma, T_{j})\bigr) \setminus \Th(T_{i}) \neq \emptyset$.
\end{enumerate}
\end{thm}

\begin{proof}
$1 \Rightarrow 2$: This is trivial. 

$2 \Rightarrow 1$: Suppose for all $i\leq k$, $\bigl(\bigcap_{\substack{j\neq i \\ j\leq k}}\Cons(\Gamma,T_{j})\bigr)\setminus\Th(T_{i})\neq\emptyset$.

Case 1: $\Gamma = \Sigma_{n}$. 

For each $i \leq k$, let $\varphi_{i} \in \bigl(\bigcap_{\substack{j \neq i \\ j \leq k}}\Cons(\Sigma_n, T_{j})\bigr) \setminus \Th(T_{i})$ and let $\theta_{i}$ be $\Pi_{n}$~sentences satisfying the following equivalences:
\[
\PA\vdash\theta_{i}\leftrightarrow\bigwedge_{\substack{j\neq i\\j \leq k}}\varphi_{j}\land 
\neg\Biggl(\Pr_{T_{i}}\biggl(\gn{\bigvee_{j\leq k}\theta_{j}}\biggr)
\prec\Pr_{T_{i}}^{\pn}\bigl(\gn{\neg\theta_{i}}\bigr)\Biggr).
\]
We show $\bigvee_{j \leq k} \theta_{j} \in \bigcap_{i \leq k} \Cons(\Sigma_{n}, T_{i}) \setminus \Th(T_{i})$.

First, we prove $T_{i} \nvdash \bigvee_{j \leq k} \theta_{j}$ for all $i \leq k$. 
Assume there is an $i^\ast \leq k$ such that $T_{i^\ast} \vdash \bigvee_{j \leq k} \theta_{j}$, then there is a $p \in \omega$ such that $\PA \vdash \Prf_{T_{i^\ast}}\bigl(\gn{\bigvee_{j \leq k} \theta_{j}},\overline{p}\bigr)$. 
Also, by Fact~\ref{rPr}.1, 
$T_{i^\ast}+\theta_{i^\ast}\vdash\forall y\leq\overline p\,\neg\Prf_{T_{i^{*}}}^{\pn}(\gn{\neg\theta_{i^{*}}},y)$. 
Then, 
\[
T_{i^\ast}+\theta_{i^\ast}\vdash\Pr_{T_{i^{*}}}\biggl(\gn{\bigvee_{j\leq k}\theta_{j}}\biggr)
\prec\Pr_{T_{i^{*}}}^{\pn}\bigl(\gn{\neg\theta_{i^{*}}}\bigr).
\] 
Hence, by the choice of $\theta_{i^\ast}$, $T_{i^\ast}+\theta_{i^\ast} \vdash \lnot \theta_{i^\ast}$. 
That is, $T_{i^\ast}\vdash \lnot \theta_{i^\ast}$. 
By our assumption, $T_{i^\ast} \vdash \bigvee_{\substack{j \neq i^\ast \\ j \leq k}} \theta_{j}$. 
For any $j \leq k$ with $j \neq i^\ast$, $\PA \vdash \theta_{j} \to \varphi_{i^\ast}$ by the choice of $\theta_{j}$. 
Therefore $T_{i^\ast} \vdash \bigvee_{\substack{j \neq i^\ast \\ j \leq k}} \theta_{j} \to \varphi_{i^\ast}$, and hence $T_{i^\ast} \vdash \varphi_{i^\ast}$. 
This contradicts the choice of~$\varphi_{i^{*}}$, which shows
$T_{i} \nvdash \bigvee_{j \leq k} \theta_{j}$. 

Next, we show $\bigvee_{j \leq k} \theta_{j} \in \Cons(\Sigma_{n}, T_{i})$ for all $i \leq k$. 
Fix an arbitrary $i\leq k$.
Let $\sigma$ be a $\Sigma_{n}$ sentence such that  $T_{i}+\bigvee_{j \leq k} \theta_{j} \vdash \sigma$. 
Then $T_{i} + \theta_{i} \vdash \sigma$. 
Therefore, by Fact~\ref{rPr}.2, there is a $q \in \omega$ such that 
$\PA+\lnot\sigma\vdash\Prf_{T_{i}}^{\pn}(\gn{\neg\theta_{i}},\overline{q})$.
Since $T_{i}\nvdash\bigvee_{j\leq k}\theta_{j}$, one has $\PA \vdash \forall y < \overline{q} \, \lnot \Prf_{T_{i}}\bigl(\gn{\bigvee_{j \leq k} \theta_{j}}, y\bigr)$. 
Hence 
\[
	\PA+ \lnot \sigma \vdash\neg\Biggl(\Pr_{T_{i}}\biggl(\gn{\bigvee_{j\leq k}\theta_{j}}\biggr)
\prec\Pr_{T_{i}}^{\pn}\bigl(\gn{\neg\theta_{i}}\bigr)\Biggr).
\] 
Therefore, $\PA + \lnot \sigma + \bigwedge_{\substack{j\neq i \\ j \leq k}}\varphi_{j} \vdash \theta_{i}$ by the choice of $\theta_{i}$. 
Since $T_{i} + \theta_{i} \vdash \sigma$, $T_{i} + \lnot \sigma + \bigwedge_{\substack{j \neq i \\ j \leq k}}\varphi_{j} \vdash \sigma$, 
so $T_{i} + \bigwedge_{\substack{j \neq i \\ j \leq k}}\varphi_{j} \vdash \sigma$. 
Since $\bigwedge_{\substack{j \neq i \\ j \leq k}}\varphi_{j} \in \Cons(\Sigma_n, T_i)$ by Lemma \ref{cCons}.1, we obtain $T_{i} \vdash \sigma$. 

Case 2: $\Gamma = \Pi_{n}$. 

For each $i \leq k$, let $\varphi_{i} \in \bigl(\bigcap_{\substack{j \neq i \\ j \leq k}}\Cons(\Pi_{n}, T_{j})\bigr) \setminus \Th(T_{i})$ and let $\theta_{i}$ be $\Sigma_{n}$~sentences satisfying 
\[
\PA \vdash \theta_{i} \leftrightarrow \bigwedge_{\substack{j \neq i \\ j \leq k}} \varphi_{j} \land
\Biggl(\Pr_{T_{i}}^{\sn}\bigl(\gn{\neg\theta_{i}}\bigr)
\prec\Pr_{T_{i}}\biggl(\gn{\bigvee_{j\leq k}\theta_{j}}\biggr)\Biggr).
\]
By almost the same argument as in Case 1, we find 
\[
\bigvee_{j \leq k} \theta_{j} \in \bigcap_{i \leq k}\Cons(\Pi_{n}, T_{i}) \setminus \Th(T_{i}).\qedhere
\]
\end{proof}

Next, we generalize Fact \ref{B2'}. 
In the case of two theories, Fact \ref{B2'} gives two sufficient conditions $\ThD(T) \nsubseteq \Th(U)$ and ``$U + \ThG(T)$ is consistent'' for $\Cons(\Gamma, T) \setminus \Th(U) \neq \emptyset$. 
These two conditions adapt straightforwardly to the case of r.e.~families of theories $\{T_{i}\}_{i\in\J}$ as the conditions $\bigcap_{i \in\J}\ThD(T_{i}) \not \subseteq \Th(U)$ and ``$U + \bigcup_{i \in\J}\ThG(T_{i})$ is consistent'', respectively. 
We can show that each of these generalized conditions implies $\bigcap_{i \in\J}\Cons(\Gamma, T_{i})\setminus \Th(U) \neq \emptyset$. 
Moreover, we found the following new condition which is also sufficient for 
$\bigcap_{i \in\J}\Cons(\Gamma, T_{i})\setminus \Th(U) \neq \emptyset$: 

\begin{description}	
	\item [B1]: There is an r.e.~set $X \subseteq\J$ such that 
\[
	\bigcap_{i \in\J\setminus X}\ThD(T_{i}) \not \subseteq \Th \biggl(U + \bigcup_{i \in X}\ThG(T_{i})\biggr).
\]
\end{description}

Here, $\bigcap_{i \in \emptyset}\ThD(T_{i})$ is the set of all sentences. 
Hence, the consistency of $U + \bigcup_{i \in\J}\ThG(T_{i})$ implies \textbf{B1} because $\J$ is r.e. 
Also, $\bigcap_{i \in\J}\ThD(T_{i}) \not \subseteq \Th(U)$ implies \textbf{B1} because $\emptyset$ is r.e. 
Therefore, the following theorem is indeed a generalization of Fact \ref{B2'}. 

\begin{thm}\label{OSC}
Let $\{T_{i}\}_{i \in\J}$ be any r.e.~family of theories. 
If condition\/~\textbf{\upshape B1} holds for $\{T_i\}_{i \in\J}$, 
then $\bigcap_{i \in\J}\Cons(\Gamma, T_{i})\setminus \Th(U) \neq \emptyset$.
\end{thm}

\begin{proof}
Let $X \subseteq\J$ be an r.e.~set such that
\[
	\bigcap_{i \in\J\setminus X}\ThD(T_{i}) \not \subseteq \Th \biggl(U + \bigcup_{i \in X}\ThG(T_{i})\biggr). 
\]
Then there is a $\Gamma^{d}$ sentence $\varphi$ satisfying the following two conditions:
\begin{enumerate}
	\item $\varphi \in \bigcap_{i \in\J\setminus X}\ThD(T_{i})$;
	\item $U + \bigcup_{i \in X}\ThG(T_{i}) + \lnot \varphi$ is consistent.
\end{enumerate}

Let $T := \PA + \bigcup_{i \in X}\ThG(T_{i})$. 
Since $X$ is an r.e.~set, $T$ is a consistent r.e.~extension of $\PA$. 
Also, since $U + \bigcup_{i \in X}\ThG(T_{i}) + \lnot \varphi \vdash U + \ThG(T) + \lnot \varphi$, we have that $U + \ThG(T) + \lnot \varphi$ is consistent. 
Therefore, there is a 
\[
\psi \in \HCons(\Gamma, T) \setminus \Th(U+\lnot \varphi)
\]
by Fact \ref{HCfB2}.

We prove $\varphi\lor\psi\in\bigcap_{i\in\J}\Cons(\Gamma,T_{i})\setminus\Th(U)$. 

Since $U + \lnot \varphi \nvdash \psi$, we obviously obtain $U \nvdash \varphi \lor \psi$. 
We prove $\varphi \lor \psi \in \Cons(\Gamma, T_{i})$ for any $i \in\J$. 
For $i \in\J\setminus X$, trivially $\varphi \lor \psi \in \Cons(\Gamma, T_{i})$ because $T_{i} \vdash \varphi$. 
For $i \in X$, let $\gamma$ be any $\Gamma$ sentence such that $T_{i} + \psi \vdash \gamma$. 
Then $T_{i} \vdash \psi \to \gamma$. 
Since $\psi\to\gamma$ is a $\Gamma$ sentence, $\PA+\ThG(T_{i})+\psi\vdash\gamma$. 
Also, since $T\vdash\PA+\ThG(T_{i})\vdash\PA$, we obtain $\PA+\ThG(T_{i})\vdash\gamma$ by the hereditary $\Gamma$-conservativity of~$\psi$. 
Thus $T_{i} \vdash \gamma$. 
Hence, $\psi \in \Cons(\Gamma, T_{i})$ and so $\varphi \lor \psi \in \Cons(\Gamma, T_{i})$ by Lemma \ref{cCons}.2. 
Therefore, $\varphi \lor \psi \in \bigcap_{i \in\J}\Cons(\Gamma, T_{i})$. 
\end{proof}

In Theorem \ref{GC2}, we will reverse the implication of Theorem~\ref{OSC} for $\Gamma = \Sigma_n$ and finite families of theories. 

We spell out a corollary of Theorem~\ref{OSC} for finite subfamilies of an infinite r.e.~family $\{T_{i}\}_{i \in \omega}$. 

\begin{cor}\label{4conditions}
Let\/ $\{T_{i}\}_{i \in \omega}$ be an infinite r.e.~family of theories and\/ $U$ be a theory. 
If there exists a set $X \subseteq \omega$ such that 
\[
\bigcap_{i \in \omega \setminus X} \Th_{\Gamma^d}(T_i) \nsubseteq \Th\biggl(U + \bigcup_{i \in X} \ThG(T_i)\biggr),
\] then for all $k \in \omega$, $\bigcap_{i \leq k}\Cons(\Gamma, T_{i})\setminus \Th(U) \neq \emptyset$. 
\end{cor}

\begin{proof}
Suppose that 
\[
\bigcap_{i \in \omega \setminus X}\ThD(T_{i}) \not \subseteq \Th \biggl(U + \bigcup_{i \in X}\ThG(T_{i})\biggr)
\]
for some $X \subseteq \omega$.
We fix a $k \in \omega$ and let $X' := X \cap {\mathrm I}_{k}$. 
Then 
\[
\bigcap_{i \in {\mathrm I}_{k} \setminus X'}\ThD(T_{i}) \not \subseteq \Th \biggl(U + \bigcup_{i \in X'}\ThG(T_{i})\biggr).
\] 
Therefore, $\bigcap_{i \leq k}\Cons(\Gamma, T_{i})\setminus \Th(U) \neq \emptyset$ by Theorem \ref{OSC}. 
\end{proof}

On the other hand, for infinite families, condition \textbf{B1} does not follow from the existence of an $X \subseteq \omega$ such that 
\[
\bigcap_{i \in \omega \setminus X} \Th_{\Gamma^d}(T_i) \nsubseteq \Th\biggl(U + \bigcup_{i \in X} \ThG(T_i)\biggr)
\]
in general. 
This will be shown in Corollary~\ref{MSCex}. 
Thus we do not know whether the assumption `$X$ is r.e.' in the statement of Theorem \ref{OSC} can be removed or not. 
Let us however show that the part of $X$ can always be played by a $\Pi_{1}$ set. 

\begin{prop}
Let $\{T_{i}\}_{i \in \omega}$ be an infinite r.e.~family of theories. 
If 
\[
\bigcap_{i \in \omega \setminus X} \Th_{\Gamma^d}(T_i) \nsubseteq \Th\biggl(U + \bigcup_{i \in X} \ThG(T_i)\biggr)
\]
for some $X \subseteq \omega$, then 
\[
\bigcap_{i \in \omega \setminus X'} \Th_{\Gamma^d}(T_i) \nsubseteq \Th\biggl(U + \bigcup_{i \in X'} \ThG(T_i)\biggr)
\]
for some $\Pi_{1}$ set $X' \subseteq \omega$. 
\end{prop}
\begin{proof}
Suppose $\varphi \in \bigcap_{i \in \omega \setminus X} \Th_{\Gamma^d}(T_i)$ and $U + \bigcup_{i \in X} \ThG(T_i) \nvdash \varphi$. 
Let $X' : = \{i \in \omega : T_i \nvdash \varphi\}$. 
Then $X'$ is a $\Pi_{1}$ set because $(T_i)_{i \in \omega}$ is a uniformly r.e.~sequence. 
Obviously $\varphi \in \bigcap_{i \in \omega \setminus X'} \Th_{\Gamma^d}(T_i)$. 
If $i \notin X$, then $T_i \vdash \varphi$, and hence $i \notin X'$. 
This means $X' \subseteq X$, and thus $U + \bigcup_{i \in X'} \ThG(T_i)$ is a subtheory of $U + \bigcup_{i \in X} \ThG(T_i)$. 
Therefore $U + \bigcup_{i \in X'} \ThG(T_i) \nvdash \varphi$. 
We conclude 
\[
\bigcap_{i \in \omega \setminus X'} \Th_{\Gamma^d}(T_i) \nsubseteq \Th\biggl(U + \bigcup_{i \in X'} \ThG(T_i)\biggr).\qedhere
\] 
\end{proof}

At last, after an auxiliary lemma, we generalize the equivalence of Fact \ref{B3'} to all finite families of theories. 

\begin{lem}\label{mililemmaduo}
Suppose $\psi\in \Cons(\Sigma_n,T)$ and $\sigma\in \Th_{\Sigma_n}(T)$.
Then $T$ proves $\neg\, ({\sim}\psi \preceq \sigma)$.
{\upshape(See Convention~\ref{simConv} for~${\sim}\psi$.)}
\end{lem}

\begin{proof}
Suppose $\psi\in \Cons(\Sigma_n,T)$ and $\sigma\in \Th_{\Sigma_n}(T)$.
Since $T \vdash \sigma$, we have $T + \psi \vdash \sigma \prec {\sim}\psi$ by Fact \ref{dual}.4, and hence,
by the $\Sigma_n$-conservativity of $\psi$, $T \vdash  \sigma \prec {\sim}\psi$.
Therefore, $T \vdash \neg\, ({\sim}\psi \preceq \sigma)$ by Fact \ref{dual}.2.
\end{proof}

\begin{thm}\label{GC2}
Let $k \in \omega$ and let $T_{0},\ldots,T_{k}$ and $U$ be theories. 
Then the following are equivalent:
\begin{enumerate}
	\item $\bigcap_{i \leq k}\Cons(\Sigma_{n},T_{i})\setminus\Th(U)\neq\emptyset$;
	\item There is an $X \subseteq {\mathrm I}_{k}$ such that
\[\bigcap_{i \in {\mathrm I}_{k} \setminus X}\ThP(T_{i}) \not \subseteq \Th \biggl(U + \bigcup_{i \in X}\ThS(T_{i})\biggr).\]
\end{enumerate}
\end{thm}

\begin{proof}

2 $\Rightarrow$ 1: Since every finite set is r.e., this follows from Theorem \ref{OSC}. 

1 $\Rightarrow$ 2: Suppose $\bigcap_{i \in {\mathrm I}_{k} \setminus X}\ThP(T_{i}) \subseteq \Th \bigl(U + \bigcup_{i \in X}\ThS(T_{i})\bigr)$ for all $X \subseteq {\mathrm I}_{k}$.
Consider any $\psi \in \bigcap_{i \leq k}\Cons(\Sigma_{n}, T_{i})$. 
We~aim to show $U \vdash \psi$.

We first prove a claim involving collections of sentences indexed by certain sequences.
These sequences take elements from~$\mathrm{I}_{k}= \{ 0,\ldots,k\}$ and they
are \emph{injective} in the sense
that no repetitions are allowed.  We fix the following notation:
\begin{itemize}
\item
$\varepsilon$ is the null sequence.
 \item
 $\subseteq$ is the prefix relation.
 \item
 $[\bseq s]$ is the set of all elements of $\bseq s$.
  (If we model sequences as functions on finite ordinals,
 we could also say: $[\bseq s]$ is the range of $\bseq s$.)
 \item
 $\bseq s i$ is the result of appending $i$ to $\bseq s$.
 We demand of course that $i\in \mathrm{I}_k \setminus [\bseq s]$.
 \item
  $\mathrm L_k$ is the set of all injective sequences with elements in~${\mathrm I}_{k}$, and $\mathrm L_k^+ := \mathrm L_k \setminus \{ \varepsilon\}$.
 \end{itemize}

\begin{cl}
There is a family\/ $(\varphi_{\bseq r})_{\bseq r\in \mathrm L_{k}^+}$ of\/ $\Sigma_n$~sentences\/ 
such that for all sequences\/~$\bseq t$ in\/~$\mathrm L_{k}$, one has:
\begin{itemize}
	\item[{\rm (i)}]  $T_{i} \vdash \varphi_{\bseq ti}$, for $i\in \mathrm{I}_k\setminus [\bseq t]$;
	\item[{\rm (ii)}] $U \vdash \bigwedge_{j \in \mathrm{I}_k \setminus [\bseq t]} \varphi_{\bseq t j} \to 
	\bigvee_{\varepsilon \neq\bseq r \subseteq \bseq t }\lnot\,({\sim}\psi \preceq \varphi_{\bseq r})$.
\end{itemize}
\end{cl}

\begin{proof}
We proceed by upward induction on~$\bseq t$ in~$(\mathrm L_{k},\mathord\subseteq)$.
Thus we consider an arbitrary element $\bseq t\in\mathrm L_{k}$ assuming
the sentences $\varphi_{\bseq r}$ satisfying condition~(i) of the claim
have already been procured for each non-null $\bseq r\subseteq\bseq t$.

Suppose $\bseq si\subseteq\bseq t$, where $\bseq s$ may be null.
Since, by condition (i) of the induction hypothesis, 
$T_{i} \vdash \varphi_{\bseq si}$, and since $\psi$ is $\Sigma_n$-conservative over $T_i$,
 we have,  by Lemma~\ref{mililemmaduo},  $T_{i} \vdash \lnot \,({\sim}\psi \preceq \varphi_{\bseq si})$.
We conclude that $T_{i} \vdash \bigvee_{\varepsilon \neq \bseq r \subseteq \bseq t} \lnot\,({\sim}\psi \preceq \varphi_{\bseq r})$ for any $i \in [\bseq t]$. 

With $X={\mathrm I}_{k} \setminus [\bseq t]$, our supposition reads 
\[\bigcap_{i\in [\bseq t]} \ThP(T_{i}) \subseteq \Th \biggl(U + \bigcup_{j\in \mathrm{I}_k\setminus [\bseq t]}\ThS(T_{j})\biggr).\]
Hence, $U + \bigcup_{j\in \mathrm{I}_k \setminus[\bseq t]}\ThS(T_{j}) \vdash \bigvee_{\varepsilon\neq \bseq r \subseteq \bseq t} \lnot \,({\sim}\psi\preceq \varphi_{\bseq r})$. 
Thus, for each $j\in {\mathrm I}_{k}\setminus [\bseq t]$, there is a $\Sigma_{n}$ sentence $\varphi_{\bseq tj}$ 
such that 
$T_{j} \vdash \varphi_{\bseq tj}$ and \[U \vdash \bigwedge_{j \in {\mathrm I}_{k} \setminus [\bseq t]} 
\varphi_{\bseq tj} \to \bigvee_{\varepsilon\neq \bseq r \subseteq \bseq t} \lnot \,({\sim}\psi \preceq \varphi_{\bseq r}).\] This shows (i) and (ii) for~$\bseq t$. 
\end{proof}

We resume the proof of 1~$\Rightarrow$~2. 
We show by downward induction on~$\bseq s$ in~$(\mathrm L_{k},\mathord\subseteq)$
that 
$U + \lnot \psi \vdash \bigvee_{\varepsilon \neq \bseq r \subseteq \bseq s } \varphi_{\bseq r}$.
The desired outcome that $U \vdash \psi$ is then
immediate from the case where $\bseq s = \varepsilon$.

Recall that 
\begin{equation}\label{A}
U+\neg\psi\vdash\lnot\,({\sim}\psi\preceq \sigma)\to\sigma
\end{equation}
for any $\Sigma_{n}$~sentence~$\sigma$ by Fact \ref{dual}.4. 

Assume $\bseq s$ is $\subseteq$-maximal, that is, $[\bseq s]={\mathrm I}_{k}$.
Then, $U \vdash \bigvee_{\varepsilon \neq \bseq r \subseteq \bseq s}\lnot\,({\sim}\psi\preceq \varphi_{\bseq r})$ by condition (ii) of the claim, 
and $U + \lnot \psi \vdash \bigvee_{\varepsilon \neq \bseq r \subseteq \bseq s} \varphi_{\bseq r}$ follows by (\ref{A}).

If $\bseq s$ is not $\subseteq$-maximal, then
the induction hypothesis yields
\[U + \lnot \psi \vdash  \varphi_{\bseq sj} \lor\bigvee_{\varepsilon \neq \bseq r \subseteq \bseq s }\varphi_{\bseq r}\] 
for each $j\in {\mathrm I}_{k} \setminus [\bseq s]$.
Hence, 
\begin{equation}\label{B}
U + \lnot \psi \vdash 
\bigwedge_{j\in {\mathrm I}_{k} \setminus [\bseq s]}\varphi_{\bseq sj}\lor
\bigvee_{\varepsilon \neq \bseq r \subseteq \bseq s }\varphi_{\bseq r}.
\end{equation}
Also, $U \vdash \bigwedge_{j\in {\mathrm I}_{k}\setminus[\bseq s]} 
\varphi_{\bseq s j} \to \bigvee_{\varepsilon \neq \bseq r \subseteq \bseq s}\lnot \,({\sim}\psi\preceq \varphi_{\bseq r})$ by condition (ii) of the claim.
By (\ref{A}), we find:
\begin{equation}\label{C}
U + \lnot\psi \vdash \bigwedge_{j\in {\mathrm I}_{k}\setminus[\bseq s]} \varphi_{\bseq s j} \to 
\bigvee_{\varepsilon\neq \bseq r \subseteq \bseq s} \varphi_{\bseq r}.
\end{equation}
Combining (\ref{B}) and (\ref{C}), we obtain $U + \lnot \psi \vdash\bigvee_{\varepsilon\neq \bseq r \subseteq \bseq s} \varphi_{\bseq r}$. 
\end{proof}

Fact~\ref{CexPi} already tells us that the $\Pi_{n}/\Sigma_{n}$-symmetric image of Theorem~\ref{GC2} fails.
Furthermore, Corollary~\ref{fCP} will show that $\bigcap_{i \leq k}\Cons(\Pi_{n}, T_{i})\setminus \Th(U)$ can never be empty.

We close this subsection with open problems concerning implications between conditions for infinite r.e.~families of theories. 
We do not know whether we can extend Theorem \ref{GC1} to  infinite families or not. 
We consider the following three conditions on infinite r.e.~families related to Guaspari's problem:

\begin{description}
	\item[G1] $\bigcap_{i \in \omega}\Cons(\Gamma, T_{i}) \setminus \Th(T_{i}) \neq \emptyset$. 
	\item[G2] For all $i \in \omega$, $\bigl(\bigcap_{\substack{j \neq i \\ j \in \omega}}\Cons(\Gamma, T_{j})\bigr) \setminus \Th(T_{i}) \neq \emptyset$. 
	\item[G3] For all $k \in \omega$, $\bigcap_{i \leq k}\Cons(\Gamma, T_{i}) \setminus \Th(T_{i})\neq \emptyset$. 
\end{description}

Recall that the versions of these three conditions for finite families are all equivalent by Theorem~\ref{GC1}.

With the help of Theorem \ref{GC1}, we obtain the following implications. 

\begin{cor}
For any infinite r.e.~family $\{T_{i}\}_{i \in \omega}$ of theories, \textbf{\upshape G1}~implies\/~\textbf{\upshape G2}, and\/ \textbf{\upshape G2}~implies\/~\textbf{\upshape G3}. 
\end{cor}

\begin{proof}
\textbf{G1} $\Rightarrow$ \textbf{G2}: This is trivial. 

\textbf{G2} $\Rightarrow$ \textbf{G3}: Suppose for all $i \in \omega$, $\bigl(\bigcap_{\substack{j \neq i \\ j \in \omega}}\Cons(\Gamma, T_{j}) \bigr) \setminus \Th(T_{i}) \neq \emptyset$. 

Let $k=0$. Then $\bigcap_{i \leq k}\Cons(\Gamma, T_{i}) \setminus \Th(T_{i}) \neq \emptyset$ by Fact \ref{exCons}. 

Let $k \geq 1$. Then $\bigl(\bigcap_{\substack{j \neq i \\ j \leq k}}\Cons(\Gamma, T_{j}) \bigr) \setminus \Th(T_{i}) \neq \emptyset$ for all $i \leq k$.  
Therefore, $\bigcap_{i \leq k}\Cons(\Gamma, T_{i}) \setminus \Th(T_{i})\neq \emptyset$ by Theorem \ref{GC1}. 
\end{proof}

In Section \ref{SecCex}, we will prove in Theorem \ref{CCC} that the implication \textbf{G3} $\Rightarrow$ \textbf{G2} does not hold in general. 

\begin{prob}\label{prob1}
Does condition \textbf{G2} imply condition \textbf{G1}?
\end{prob}


In connection with Bennet's analysis, we have dealt with the following four conditions:

\begin{description}
	\item [B1] There exists an r.e.~set $X \subseteq \J$ such that 
\[
\bigcap_{i \in \J\setminus X} \Th_{\Gamma^d}(T_i) \nsubseteq \Th\biggl(U + \bigcup_{i \in X} \ThG(T_i)\biggr).
\] 
	\item [B2] There exists a set $X \subseteq \J$ such that 
\[
\bigcap_{i \in \J\setminus X} \Th_{\Gamma^d}(T_i) \nsubseteq \Th\biggl(U + \bigcup_{i \in X} \ThG(T_i)\biggr).
\] 
	\item [B3] $\bigcap_{i \in \J}\Cons(\Gamma, T_{i})\setminus \Th(U) \neq \emptyset$. 
	\item [B4] $\bigcap_{i \leq k}\Cons(\Gamma, T_{i}) \setminus \Th(U) \neq \emptyset$ for all $k \in \J$. 
\end{description}

For finite families of theories, of course, we have \textbf{B1} $\Leftrightarrow$ \textbf{B2} $\Rightarrow$ \textbf{B3} $\Leftrightarrow$ \textbf{B4} by Theorem \ref{OSC}. 
Moreover, in the case of $\Gamma = \Sigma_{n}$, \textbf{B3} $\Rightarrow$ \textbf{B2} by Theorem \ref{GC2}. 

For arbitrary r.e.~families we have the following implications:

\begin{center}
\begin{tikzpicture}
\node (one) at (0,0) {\large \textbf{B1}};
\node (two) at (3,1) {\large \textbf{B2}};
\node (three) at (3,-1){\large \textbf{B3}};
\node (four) at (6,0){\large \textbf{B4}};
\node[below] at (0.5, -0.5) {Theorem \ref{OSC}};
\node[above] at (0.75, 0.5) {trivial};
\node[below] at (5.25, -0.5) {trivial};
\node[above] at (5.5, 0.5) {Corollary \ref{4conditions}};
\draw [->, double distance=1pt] (one)--(two);
\draw [->, double distance=1pt] (one)--(three);
\draw [->, double distance=1pt] (two)--(four);
\draw [->, double distance=1pt] (three)--(four);
\end{tikzpicture}
\end{center}


In Section \ref{SecCex}, we will show that neither \textbf{B2}~$\Rightarrow$~\textbf{B1} nor \textbf{B4}~$\Rightarrow$~\textbf{B3} holds~--- see Corollaries \ref{MSCex} and \ref{counterexample1}, respectively.
For $\Gamma = \Pi_{n}$, Fact \ref{CexPi} gives a counterexample to the implication \textbf{B3} $\Rightarrow$ \textbf{B2}. 
Therefore, for $\Gamma = \Pi_{n}$, neither \textbf{B3} $\Rightarrow$ \textbf{B1} nor \textbf{B4} $\Rightarrow$ \textbf{B2} holds. 
We do not know whether the other implications hold or not. 

\begin{prob}
Does the implication \textbf{B2} $\Rightarrow$ \textbf{B3} hold? 
Also, for $\Gamma = \Sigma_{n}$, do the implications \textbf{B3}~$\Rightarrow$~\textbf{B1}, \textbf{B3}~$\Rightarrow$~\textbf{B2} and \textbf{B4}~$\Rightarrow$~\textbf{B2} hold?
\end{prob}

Notice that the potential implication \textbf{B2}~$\Rightarrow$~\textbf{B3} would strengthen Theorem~\ref{OSC} and Corollary~\ref{4conditions}. 
For $\Gamma = \Sigma_{n}$, the implications \textbf{B3}~$\Rightarrow$~\textbf{B1} and \textbf{B3}~$\Rightarrow$~\textbf{B2} would extend Theorem \ref{GC2} to infinite r.e.~families of theories.

\subsection{Hereditarily $\Gamma$-conservative sentences}\label{HGCs}

In this subsection, we generalize Facts \ref{HCfB1} and \ref{HCfB2}. 
In our proofs, we use the following lemma.

\begin{lem}\label{lHC}
Let $T$ and $U$ be any theories. If $\ThG(U) \subseteq \ThG(T)$, then $\HCons(\Gamma, T) \subseteq \HCons(\Gamma, U)$. 
\end{lem}

\begin{proof}
Let $\psi \in \HCons(\Gamma, T)$. 
To prove $\psi \in \HCons(\Gamma, U)$, let $S$ be any theory and let $\varphi$ be any $\Gamma$ sentence such that $U \vdash S \vdash \PA$ and $S + \psi \vdash \varphi$. 
Since $\psi \to \varphi$ is $\Gamma$ sentence, we have $\PA + \ThG(S) + \psi \vdash \varphi$. 
Since $T \vdash \PA + \ThG(T) \vdash \PA + \ThG(S) \vdash \PA$, we obtain $\PA + \ThG(S) \vdash \varphi$. 
Therefore, $S \vdash \varphi$. 
\end{proof}

First, we generalize Fact \ref{HCfB2} to the case of r.e.~families of theories by using Fact \ref{HCfB2} itself. 

\begin{thm}\label{GHC2}
For any r.e.~family $\{T_{i}\}_{i \in \J}$ of theories and for any theory\/~$U$, the following are equivalent:
\begin{enumerate}
	\item $\bigcap_{i \in \J}\HCons(\Gamma, T_{i})\setminus \Th(U) \neq \emptyset$;
	\item For all $k \in \J$, $\bigcap_{i \leq k}\HCons(\Gamma, T_{i})\setminus \Th(U) \neq \emptyset$;
	\item $U + \bigcup_{i \in \J}\ThG(T_{i})$ is consistent.
\end{enumerate}
\end{thm}

\begin{proof}
1 $\Rightarrow$ 2: This is trivial.

2 $\Rightarrow$ 3: Suppose $U+\bigcup_{i\in\J}\ThG(T_{i})$ is inconsistent. 
Then there is a $k\in\J$ such that $U+\bigcup_{i\leq k}\ThG(T_{i})$ is inconsistent. 
Thus, there are $\Gamma$ sentences $\varphi_{0},\ldots,\varphi_{k}$ such that $T_{i}\vdash\varphi_{i}$ for each $i\leq k$ and $U\vdash\bigvee_{i\leq k}\lnot\varphi_{i}$. 
For any $\Gamma^{d}$~sentence $\psi$ and $i\leq k$,
let $T_{i}^{\psi} : = \PA + \varphi_{i} \lor \lnot \psi$.
By the choice of $\varphi_{i}$, $T_{i}\vdash T_{i}^{\psi}\vdash\PA$ for each 
$i\leq k$.
Suppose $\psi\in\bigcap_{i \leq k}\HCons(\Gamma,T_{i})$. 
Let us show $U\vdash\psi$:
%
%
%

Consider any $i\leq k$.
As $T_{i}^{\psi}+\psi\vdash\varphi_{i}$, 
we must have $T_{i}^{\psi}\vdash\varphi_{i}$, for $\psi\in\HCons(\Gamma,T_{i})\subseteq\Cons(\Gamma,T_{i}^{\psi})$.
Therefore $\PA+\lnot\psi\vdash\varphi_{i}$.
Thus $\PA+\lnot \psi\vdash\bigwedge_{i \leq k}\varphi_{i}$,
that is, $\PA\vdash\bigvee_{i\leq k}\lnot\varphi_{i}\to\psi$. 
Since $U \vdash \bigvee_{i \leq k} \lnot \varphi_{i}$, one has $U \vdash \psi$. 

We~have shown 
$\bigcap_{i\leq k}\HCons(\Gamma, T_{i})\subseteq\Th(U)$. 
In~other words, $\bigcap_{i \leq k}\HCons(\Gamma, T_{i})\setminus \Th(U) = \emptyset$.

3 $\Rightarrow$ 1: Suppose $U + \bigcup_{i \in \J}\ThG(T_{i})$ is consistent. 
Let $T^{+}$ be the theory $\PA + \bigcup_{i \in \J}\ThG(T_{i})$. 
Then $T^{+}$ is a consistent r.e.~extension of $\PA$. 
Since $U + T^{+} \vdash U + \ThG(T^{+})$, $U + \ThG(T^{+})$ is also consistent. 
Therefore, $\HCons(\Gamma, T^{+}) \setminus \Th(U) \neq \emptyset$ by Fact \ref{HCfB2}.
Since $\ThG(T_{i})\subseteq\ThG(T^{+})$ for each $i\in\J$, we have $\HCons(\Gamma, T^{+})\subseteq\bigcap_{i \in \J}\HCons(\Gamma,T_{i})$ by Lemma~\ref{lHC}. 
We conclude 
\[
\bigcap_{i \in \J}\HCons(\Gamma, T_{i})\setminus \Th(U) \neq \emptyset.\qedhere
\]
\end{proof}

\begin{rem}
In the case of $\Cons(\Gamma,T)$, the equivalence of the conditions corresponding to 1 and 2 in Theorem~\ref{GHC2} does not generally hold (Corollary \ref{counterexample1}). 
\end{rem}

Secondly, we generalize Fact \ref{HCfB1}. 
For this purpose, we write down several equivalents of $\bigcap_{i \in \J}\HCons(\Gamma, T_{i}) \setminus \Th(T_{i})\neq \emptyset$. 

\begin{prop}\label{infHCons1}
For any r.e.~family $\{T_{i}\}_{i \in \J}$ of theories, the following are equivalent:
\begin{enumerate}
	\item $\bigcap_{i \in \J}\HCons(\Gamma, T_{i}) \setminus \Th(T_{i}) \neq \emptyset$; 
	\item For all $k \in \J$, $\bigcap_{i \leq k}\HCons(\Gamma, T_{i}) \setminus \Th(T_{i})\neq \emptyset$;
	\item For all $k \in \J$ and $i \leq k$, $\bigl( \bigcap_{\substack{j \neq i \\ j \leq k}}\HCons(\Gamma, T_{j}) \bigr)\setminus \Th(T_{i}) \neq \emptyset$; 
	\item For all $i \in \J$, $T_{i} + \bigcup_{\substack{j \neq i \\ j \in \J}}\ThG(T_{j})$ is consistent. 
\end{enumerate}
\end{prop}

\begin{proof}
1 $\Rightarrow$ 2 and 2 $\Rightarrow$ 3 are trivial. 

3 $\Rightarrow$ 4: 
Let $i\leq k\in\J$. 
By Theorem \ref{GHC2}, $T_{i} + \bigcup_{\substack{j \neq i \\ j \leq k}}\ThG(T_{j})$ is consistent. 
Since $k \geq i$ is arbitrary, $T_{i} + \bigcup_{\substack{j \neq i \\ j \in \J}}\ThG(T_{j})$ is consistent. 

4 $\Rightarrow$ 1: 
Let $T^{+} := \PA + \bigcup_{i \in \J}\ThG(T_{i})$. 
Then, for each $i \in \J$, $T^{+} + T_{i}$ is consistent by clause 4. 
Hence, the r.e.~set $\bigcup_{i \in \J}\Th(T_{i})$ is pointwise consistent with $T^{+}$. 
By Fact \ref{Lin84}, we have
\[
	\HCons(\Gamma, T^{+}) \setminus \bigcup_{i \in \J}\Th(T_{i}) \neq \emptyset.
\]
Since $\ThG(T_{i}) \subseteq \ThG(T^{+})$ for all $i \in \J$, we have $\HCons(\Gamma, T^{+}) \subseteq \bigcap_{i \in \J} \HCons(\Gamma, T_{i})$ by Lemma~\ref{lHC}. 
Therefore, 
\[
	\bigcap_{i \in \J} \HCons(\Gamma, T_{i}) \setminus \bigcup_{i \in \J}\Th(T_{i}) \neq \emptyset.\qedhere
\]
\end{proof}


The following corrollary which is a generalization of Fact \ref{HCfB1} immediately follows from 3~$\Rightarrow$~1 of Proposition \ref{infHCons1}. 

\begin{cor}\label{GHC1}
For any r.e.~family $\{T_{i}\}_{i \in \J}$ of theories, the following are equivalent:
\begin{enumerate}
	\item $\bigcap_{i \in \J}\HCons(\Gamma, T_{i}) \setminus \Th(T_{i})\neq \emptyset$; 
	\item For all $i \in \J$, $\bigl( \bigcap_{\substack{j \neq i \\ j \in \J}}\HCons(\Gamma, T_{j}) \bigr)\setminus \Th(T_{i}) \neq \emptyset$.\qed
\end{enumerate}
\end{cor}

\begin{rem}
In the case of $\Cons(\Gamma, T)$, the equivalence of the conditions corresponding to 1 and 2 in Proposition \ref{infHCons1} does not generally hold (Corollary \ref{CfC}). 
We do not know whether the equivalence corresponding to 1~$\Leftrightarrow$~2 in Corollary~\ref{GHC1} holds (Problem~\ref{prob1}). 
\end{rem}


\section{$\Pi_{n}$-conservative $\Sigma_{n}$ sentences for finitely many theories}\label{Picons}

In this section, we prove that for any two theories $T$ and $U$, there exists a $U$-unprovable $\Sigma_n$ sentence which is $\Pi_n$-conservative over $T$. 
This solves Bennet's Problem \ref{BenProb}. 
As a consequence, in the case of $\Gamma = \Pi_n$, we give an affirmative answer to Guaspari's Problem \ref{Mainprob} restricted to finite families of theories~--- this also settles Problem~\ref{MisProb}. 

\begin{thm}\label{Thm_Pi}
For any theories\/ $T$ and~\/$U$, one has\/ $\Cons(\Pi_n, T) \setminus \Th(U) \neq \emptyset$. 
\end{thm}
\begin{proof}
Let $\Prv_U(x)$ be a standard $\Sigma_1$ provability predicate for $U$ satisfying $U \nvdash\Prv_{U}^{\Sigma_{n}}(\gn{0=1})$ as in Fact~\ref{nCon}. 
If the $\Sigma_n$~sentence $\Prv_{U}^{\Sigma_{n}}(\gn{0=1})$ is $\Pi_n$-conservative over~$T$, then $\Prv_{U}^{\Sigma_{n}}(\gn{0=1})\in\Cons(\Pi_{n},T)\setminus\Th(U)$, 
so we are done. 

If $\Prv_{U}^{\Sigma_{n}}(\gn{0=1})$ is not $\Pi_n$-conservative over $T$, then there exists a $\sn$~sentence $\alpha$ such that $T +\alpha$ is consistent and 
$T+\alpha\vdash\neg \Prv_{U}^{\Sigma_{n}}(\gn{0=1})$.
There exists a $\Sigma_{n}$~sentence $\sigma$ such that 
\[
\sigma \in \HCons(\Pi_{n}, T+\alpha) \setminus \Th(T +\alpha)
\] 
by Fact \ref{exHCons}. 
Since $T+\alpha\nvdash\sigma$, there exists a model $M$ of $T+\alpha$ such that $M\models \neg \sigma$. 
Let $U^{+} := U + \ThS(M)$. 
Even though $U^{+}$ likely fails to be a theory in the sense of the present paper
in~view of possible lack of recursive enumerability, it~is still going to be useful.

First, we show that $U^{+}$ is consistent. 
Suppose, towards contradiction, that $U^{+}$ is inconsistent. 
Then, there exists a $\Sigma_{n}$ sentence $\psi$ such that $U \vdash \lnot \psi$ and $M \models \psi$. 
Then $\PA \vdash \Prv_U(\gn{\lnot \psi})$. 
Since $\psi$ is a $\Sigma_{n}$ sentence, $\PA\vdash\psi\to\Pr_U^{\sn}(\gn{\psi})$, and hence $\PA \vdash \psi \to \Prv_{U}^{\sn}(\gn{0=1})$. 
Therefore $T +\alpha\vdash \lnot \psi$. 
Since $M$ is a model of $T +\alpha$,
we find that $\lnot \psi$ is true in $M$, a contradiction. 
Therefore $U^+$ is consistent. 

As~usual, we assume that both $\alpha$ and~$\sigma$ are written in the form of a $\pnmo$~formula following a single existential quantifier. 
Let $\varphi$ and $\varphi^*$ be $\Sigma_n$ sentences satisfying
\[
\PA\vdash\varphi\leftrightarrow(\sigma\preceq\alpha)\vee\bigl(\sigma\preceq\Pr_{U}^{\pnmo}(\gn{\varphi})\bigr), \ \text{and}
\]
\[
	\varphi^*\equiv(\alpha\prec\sigma)\wedge\bigl(\Pr_{U}^{\pnmo}(\gn{\varphi})\prec\sigma\bigr). 
\]
Then we have $\PA \vdash \neg (\varphi \land \varphi^*)$ and $\PA \vdash \sigma \to (\varphi \lor \varphi^*)$ by Fact~\ref{dual}.  

We prove $\varphi \in\Cons(\Pi_{n}, T) \setminus \Th(U)$. 

Suppose one had $U \vdash \varphi$. 
Then $\PA\vdash\Pr_{U}^{\pnmo}(\gn{\varphi})$. 
Since $M \models \alpha\land \neg\sigma$, we obtain $M \models \varphi^{*}$
by Fact~\ref{dual}.4. 
Then $\varphi^{*} \in \ThS(M)$. 
Thus $U^{+} \vdash \varphi \land \varphi^{*}$. 
This contradicts the consistency of $U^+$. 
Hence $U \nvdash \varphi$. 

By the definition of $\varphi^*$, $\PA \vdash \varphi^* \to\alpha\land\Pr_{U}^{\pnmo}(\gn{\varphi})$ in view of Fact~\ref{dual}.1. 
Since $\varphi^*$ is $\Sigma_n$, we have $\PA \vdash \varphi^* \to \Pr_{U}^{\pnmo}(\gn{\varphi^*})$ by Fact~\ref{rPr}. 
Hence $\PA \vdash \varphi^* \to\alpha\land \Pr_{U}^{\pnmo}(\gn{0=1})$. 
Since $T +\alpha\vdash \neg \Pr_{U}^{\pnmo}(\gn{0=1})$, we have $T \vdash \neg \varphi^*$. 
Therefore $T + \sigma \vdash \varphi$. 
Finally, from $\sigma\in\HCons(\Pi_{n},T+\alpha)$ we conclude $\sigma\in\Cons(\Pi_{n},T)$, and hence $\varphi\in\Cons(\Pi_{n},T)$ by Lemma \ref{cCons}.2.
\end{proof}

\begin{cor}\label{fCP}
Let $k$ be any natural number. 
Then for any theories $T_{0}, \ldots, T_{k}$ and\/ $U$, we have $\bigcap_{i \leq k}\Cons(\Pi_{n}, T_{i})\setminus \Th(U) \neq \emptyset$. 
\end{cor}

\begin{proof}
We argue by induction on $k$. 
For $k = 0$, this is Theorem \ref{Thm_Pi}. 
Suppose that the statement holds for $k$, and let $T_{0}, \ldots, T_{k}, T_{k+1}, U$ be any theories. 
Then by induction hypothesis, there exists a sentence $\varphi$ contained in $\bigcap_{i \leq k}\Cons(\Pi_{n}, T_{i})\setminus \Th(U)$. 
Since $U + \neg \varphi$ is consistent, there exists a sentence $\psi$ contained in $\Cons(\Pi_n, T_{k+1}) \setminus \Th(U + \neg \varphi)$ by Theorem \ref{Thm_Pi}. 
Then $\varphi \lor \psi$ is in the set $\bigcap_{i \leq {k+1}}\Cons(\Pi_{n}, T_{i})\setminus \Th(U)$ by Lemma \ref{cCons}.2. 
\end{proof}

By combining this corollary with Theorem \ref{GC1}, we solve Guaspari's Problem~\ref{Mainprob} for finite families of theories and $\Gamma = \Pi_{n}$. 

\begin{cor}\label{GuaPP}
Let $\{T_i\}_{i \leq k}$ be any finite family of theories. 
Then we have $\bigcap_{i \leq k}\Cons(\Pi_n, T_{i}) \setminus \Th(T_{i}) \neq \emptyset$. 
\qed
\end{cor}

Thus, in contrast to the case of $\Gamma = \Sigma_n$ (see Theorem \ref{GC2}), every finite family of theories admits a $\Sigma_n$ sentence which is simultaneously nontrivially $\Pi_n$-conservative over all theories in the family. 

Notice that our proof of Theorem \ref{Thm_Pi} does not provide an effective procedure for finding an element of $\Cons(\Pi_{n}, T) \setminus \Th(U)$ from (indices for) $T$ and $U$.

\begin{prob}\label{ec}
Given $T$ and $U$, can we effectively find a $\Sigma_{n}$ sentence $\varphi$ such that
\[
\varphi \in \Cons(\Pi_n, T) \setminus \Th(U)\;?
\]
\end{prob}

If Problem~\ref{ec} has an affirmative answer, then for any finite family of theories, we can effectively find a $\Sigma_{n}$ sentence which is simultaneously nontrivially $\Pi_n$-conservative over all theories in the family by the proofs of Corollary \ref{fCP} and Theorem \ref{GC1}. 

Finally, we propose the problem asking whether Theorem \ref{Thm_Pi} can be strengthened in the spirit of Fact~\ref{Lin84}. 

\begin{prob}
For a theory $T$ and an r.e.~set $X$ of sentences that is pointwise consistent with $\PA$, must one have $\Cons(\Pi_n, T) \setminus X \neq \emptyset$~?
\end{prob}

Needless to say, one cannot hope for a positive answer with $\Cons(\Sigma_n, T)$ instead of $\Cons(\Pi_n, T)$ (see comments just below Fact~\ref{B3'}).

\section{Counterexamples}\label{SecCex}

In connection with Guaspari's and Bennet's problems, we have studied a number of conditions on finite and infinite r.e.~families of theories.
In this section, we show the failure of implication between several of those conditions.

As we have already mentioned, Corollary~\ref{GHC1} reduces the investigation of the condition $\bigcap_{i \leq k+1}\HCons(\Gamma, T_{i}) \setminus \Th(T_{i})\neq \emptyset$ to that of conditions of the form
$\bigcap_{i \leq k}\HCons(\Gamma, T_{i})\setminus \Th(U) \neq \emptyset$.
One may ask whether it can be further reduced to  some simple conditions such as $\HCons(\Gamma, T) \setminus \Th(U) \neq \emptyset$. 
This does not appear to be the case:  

\begin{thm}\label{CexHCons1}
For any $k \geq 1$, there are theories $T_{0},\ldots,T_{k+1}$ satisfying the following conditions:
\begin{enumerate}
	\item $\bigcap_{i \leq k+1}\HCons(\Gamma, T_{i})\setminus\Th(T_{i})=\emptyset$;
	\item For all distinct $i_{0},i_{1}\leq k+1$, $\bigl(\bigcap_{\substack{j \neq i_{0}, i_{1} \\ j \leq k+1}}\HCons(\Gamma, T_{j})\bigr) \setminus \Th(T_{i_{1}}) \neq \emptyset$.
\end{enumerate}
\end{thm}

\begin{proof}
It suffices to find theories $T_{0},\ldots,T_{k+1}$ satisfying the following two conditions:
\begin{itemize}
	\item[(i)] $T_{k+1} + \bigcup_{i \leq k}\ThG(T_{i})$ is inconsistent;
	\item[(ii)] For any $i \leq k+1$, $\bigcup_{\substack{j \neq i \\ j \leq k+1}}T_{j}$ is consistent.
\end{itemize}
This is because (i) implies that $\bigcap_{i \leq k}\HCons(\Gamma, T_{i})\setminus \Th(T_{k+1}) = \emptyset$ by Theorem \ref{GHC2}. 
Therefore, $\bigcap_{i \leq k+1}\HCons(\Gamma, T_{i}) \setminus \Th(T_{i})= \emptyset$.
Moreover, (ii) implies that for all distinct $i_{0},i_{1}\leq k+1$, $T_{i_{1}} + \bigcup_{\substack{j \neq i_{0},  i_{1} \\ j \leq k+1}}\ThG(T_{j})$ are consistent. 
Therefore, for all distinct $i_{0},i_{1}\leq k+1$, $\bigl(\bigcap_{\substack{j \neq i_{0}, i_{1} \\ j \leq k+1}}\HCons(\Gamma, T_{j})\bigr) \setminus \Th(T_{i_{1}}) \neq \emptyset$ by Theorem \ref{GHC2}. 

Let $\xi_{0},\ldots,\xi_{k}$ be $\Gamma$~sentences such that
$\PA+\bigwedge_{i\in X}\xi_{i}+\bigwedge_{i\in\Ik\setminus X}\neg\xi_{i}$ is consistent for each $X\subseteq\Ik$ (see Lindstr\"om \cite[Theorem 2.9]{Lin}). 
Let $T_{k+1} := \PA + \bigvee_{i \leq k} \lnot \xi_{i}$ and for each $i \leq k$, 
let $T_{i} := \PA + \xi_{i}$. 
Then $T_{k+1} + \bigcup_{i \leq k}\ThG(T_{i})$ is obviously inconsistent.
Moreover, for $i \leq k$, 
$\bigcup_{\substack{j \neq i \\ j \leq k+1}}T_{j}$ is deductively equivalent to 
$\PA + \bigwedge_{\substack{j \neq i \\ j \leq k}}\xi_{j} + \lnot \xi_{i}$. 
Hence, $\bigcup_{\substack{j \neq i \\ j \leq k+1}}T_{j}$ is consistent by the choice of~$\xi_{i}$. 
For $i=k+1$, $\bigcup_{\substack{j \neq i \\ j \leq k+1}}T_{j}$ is deductively equivalent to 
$\PA + \bigwedge_{j \leq k}\xi_{j}$. 
Hence, $\bigcup_{\substack{j \neq i \\ j \leq k+1}}T_{j}$ is also consistent by the choice of~$\xi_{i}$. 
Therefore, for any $i \leq k+1$, $\bigcup_{\substack{j \neq i \\ j \leq k+1}}T_{j}$ is consistent.
\end{proof}

Moreover, from the proof of Theorem \ref{CexHCons1}, we obtain the following corollary. 

\begin{cor}
For any $k \geq 1$, there are theories $T_{0},\ldots,T_{k}$ and\/ $U$ satisfying the following conditions:
\begin{enumerate}
	\item $\bigcap_{i \leq k}\HCons(\Gamma, T_{i})\setminus \Th(U) = \emptyset$; 
	\item For all $i \leq k$, $\bigcap_{\substack{j \neq i \\ j \leq k}}\HCons(\Gamma, T_{j})\setminus\Th(U) \neq \emptyset$.\qed
\end{enumerate}
\end{cor}

For $\Sigma_n$-conservative sentences, we have a result similar to Theorem \ref{CexHCons1}. 
Observe that the $\Gamma=\Sigma_{n}$ half of Theorem~\ref{CexHCons1} is strengthened by the following one.

\begin{thm}\label{CE1}
For any $k \geq 1$, there are theories $T_{0},\ldots,T_{k+1}$ such that
\begin{enumerate}
	\item $\bigcap_{i \leq k+1}\Cons(\Sigma_{n}, T_{i}) \setminus \Th(T_{i})= \emptyset$;
	\item For all distinct $i_{0},i_{1}\leq k+1$, $\bigl(\bigcap_{\substack{j \neq i_{0}, i_{1} \\ j \leq k+1}}\HCons(\Sigma_{n},T_{j})\bigr) \setminus\Th(T_{i_{1}}) \neq \emptyset$.
\end{enumerate}
\end{thm}

\begin{proof}
%
As~a first step, we produce a suite $\tau,\sigma_{0}, \ldots, \sigma_{k}$
of sentences with certain desirable properties.
These sentences will then serve as building blocks for the construction
of theories $T_{i}$ instantiating the theorem.

By Fact \ref{exCons}, fix a $\Sigma_{n}$ sentence $\tau$ such that $\tau \in \Cons(\Pi_{n}, \PA) \setminus \Th(\PA)$. 
Let $\sigma_{0}, \ldots, \sigma_{k}$ be $\Sigma_{n}$ sentences satisfying
\begin{eqnarray*}
\PA \vdash \sigma_{i} &\leftrightarrow& \bigwedge_{j<i} \Prv_{\PA + \lnot \tau}^{\Pi_{n-1}}(\gn{\lnot \sigma_{i}}) \prec \Prv_{\PA + \lnot \tau}^{\Pi_{n-1}}(\gn{\lnot \sigma_{j}})\\
&& \qquad\land\bigwedge_{i < j \leq k} \Prv_{\PA + \lnot \tau}^{\Pi_{n-1}}(\gn{\lnot \sigma_{i}}) \preceq \Prv_{\PA + \lnot \tau}^{\Pi_{n-1}}(\gn{\lnot \sigma_{j}}).
\end{eqnarray*}

We show that the $\Sigma_{n}$ sentences  $\sigma_{0},\ldots,\sigma_{k}$ and $\tau$ satisfy (a)--(c) below: 
\begin{itemize}
	\item[(a)] For each $i \leq k$, $\PA + \lnot \tau + \sigma_{i}$ is consistent;  
	\item[(b)] For each distinct $i, j \leq k$, $\PA \vdash \lnot(\sigma_{i} \land \sigma_{j})$; 
	\item[(c)] $\Cons(\Pi_{n},\PA)\owns\bigvee_{i\leq k}\sigma_{i}$. 
\end{itemize}


(a): 
Suppose there existed an $l\leq k$ such that $\PA+\lnot\tau\vdash\lnot\sigma_{l}$. 
For some $q\in\omega$, the $\Delta_{1}$~sentence 
$\Prf_{\PA+\lnot\tau}(\gn{0=0\to\neg\sigma_{l}},\overline q)$ must then be true.

Reason in $\PA+\neg\tau$:
Since $0=0$ is a true $\pnmo$~sentence,
$\Prf_{\PA+\lnot\tau}^{\pnmo}(\gn{\neg\sigma_{l}},\overline q)$ holds.
Thus we can fix the smallest $p\leq q$ for which there is an $i \leq k$ such that 
\begin{align}\label{n1'}
\Prf_{\PA+\lnot\tau}^{\pnmo}(\gn{\neg\sigma_{i}},\overline p).
\end{align}
Consider the smallest~$i$ satisfying~(\ref{n1'}).
Then, for all $m\leq p$ and $j<i$, as~well as for all $m<p$ and $j>i$ $(j\leq k)$, one has $\neg\Prf_{\PA+\lnot\tau}^{\pnmo}(\gn{\neg\sigma_{j}},\overline m)$.
Therefore, $\sigma_{i}$ holds. 
On~the other hand, since $p$ is standard, 
Fact~\ref{rPr}.1 applied to~(\ref{n1'}) shows $\lnot\sigma_i$. 
This is a contradiction in $\PA+\neg\tau$.

But $\tau$ was chosen to be consistent with $\PA$. 
Hence the theories
$\PA + \lnot \tau + \sigma_{l}$ are consistent for all $l\leq k$. 

(b): 
Let $i < j \leq k$. Then 
\begin{eqnarray*}
\PA \vdash \sigma_{i} \land \sigma_{j} &\to& \Prv_{\PA + \lnot \tau}^{\Pi_{n-1}}(\gn{\lnot \sigma_{i}}) \preceq \Prv_{\PA + \lnot \tau}^{\Pi_{n-1}}(\gn{\lnot \sigma_{j}})\\
&& \qquad{}\land\Prv_{\PA + \lnot \tau}^{\Pi_{n-1}}(\gn{\lnot \sigma_{j}}) \prec \Prv_{\PA + \lnot \tau}^{\Pi_{n-1}}(\gn{\lnot \sigma_{i}}).
\end{eqnarray*}
Therefore, $\PA \vdash \lnot (\sigma_{i} \land \sigma_{j})$ by Fact \ref{dual}.2. 


(c): 
We shall show $\PA + \Prv_{\PA}^{\Pi_{n-1}}(\gn{0=1}) \vdash \bigvee_{i \leq k} \sigma_{i}$. 
Then, we have $\PA + \Prv_{\PA}^{\Sigma_{n}}(\gn{0=1}) \vdash \bigvee_{i \leq k} \sigma_{i}$ and therefore $\bigvee_{i \leq k} \sigma_{i} \in \Cons(\Pi_{n}, \PA)$ by Proposition \ref{lPTS} and Lemma \ref{cCons}.2. 

Reason in $\PA + \Prv_{\PA}^{\Pi_{n-1}}(\gn{0=1})$:
From $\Prv_{\PA}^{\Pi_{n-1}}(\gn{0=1})$, we have $\bigwedge_{j \leq k}\Prv_{\PA + \lnot \tau}^{\Pi_{n-1}}(\gn{\lnot \sigma_{j}})$.
Choose the smallest $i\leq k$ such that 
\[
\Prv_{\PA + \lnot \tau}^{\Pi_{n-1}}(\gn{\lnot \sigma_{i}}) \preceq \Prv_{\PA + \lnot \tau}^{\Pi_{n-1}}(\gn{\lnot \sigma_{j}})
\]
for all $j\leq k$. Then $\sigma_{i}$ must hold. 


Armed with the freshly selected sentences, we~are now prepared to say what our
theories~$T_{i}$ are.
Let
\[
	T_{k+1} := \PA + \lnot \tau \ \ \text{and} \ \ T_{i} := \PA + (\tau \land \sigma_{i}) \lor \bigvee_{\substack{j \neq i \\ j \leq k}}\sigma_{j} \ \text{for} \ i \leq k.
\]

We prove that the theories $T_{0},\ldots, T_{k+1}$ satisfy the following three conditions:  
\begin{itemize}
	\item[(i)] For each $i \leq k$, $\ThP(T_{i}) \subseteq \ThP(\PA)$;  
	\item[(ii)] $T_{k+1} + \bigcup_{i \leq k}\ThS(T_{i})$ is inconsistent; 
	\item[(iii)] For each $i \leq k+1$, the theory $\bigcup_{\substack{j \neq i \\ j \leq k+1}}T_{j}$ is consistent. 
\end{itemize}

(i): By (c) and Lemma \ref{cCons}.1, we have $\tau \land \bigvee_{j \leq k}\sigma_{j} \in \Cons(\Pi_{n}, \PA)$. 
For each $i \leq k$, since $\PA + \tau \land \bigvee_{j \leq k}\sigma_{j} \vdash T_{i}$, we have $\ThP(T_{i}) \subseteq \ThP(\PA)$. 

(ii): By the choice of $T_{i}$, we know
\[
\PA + \bigcup_{i \leq k} \ThS(T_{i}) \vdash \bigwedge_{i \leq k}\Biggl((\tau \land \sigma_{i}) \lor \bigvee_{\substack{j \neq i \\ j \leq k}}\sigma_{j}\Biggr).
\]
Since $T_{k+1} \vdash \lnot \tau$, 
\[
T_{k+1} + \bigcup_{i \leq k} \ThS(T_{i}) \vdash \bigwedge_{i \leq k} \bigvee_{\substack{j \neq i \\ j \leq k}}\sigma_{j}.
\]
Therefore, $T_{k+1} + \bigcup_{i\leq k}\ThS(T_{i})$ is inconsistent by (b). 

(iii): Suppose $i = k+1$. 
For each $j \leq k$, $\PA + \tau \land \bigvee_{l \leq k} \sigma_{l} \vdash T_{j}$ as argued in the proof of~(i). 
Furthermore, since $\tau\land\bigvee_{l \leq k}\sigma_{l} \in \Cons(\Pi_{n}, \PA)$, $\PA + \tau \land \bigvee_{l \leq k}\sigma_{l}$ is consistent. Therefore, $\bigcup_{j \leq k}T_{j}$ is consistent. 

Suppose $i \leq k$. Then, $\PA + \lnot \tau + \sigma_{i} \vdash \bigcup_{\substack{j \neq i \\ j \leq k+1}} T_{j}$. Therefore, $\bigcup_{\substack{j \neq i \\ j \leq k+1}} T_{j}$ is consistent by (a). 

At last, we show that the theories $T_{0},\ldots,T_{k+1}$ satisfy the conditions of the theorem:

1: We are going to show $\bigcap_{i\leq k}\Cons(\Sigma_{n},T_{i})\setminus\Th(T_{k+1})=\emptyset$.
According to Theorem~\ref{GC2}, it suffices to verify that for each $X \subseteq \mathrm{I}_{k}$, one has
\[
	\bigcap_{i \in \mathrm{I}_{k} \setminus X}\ThP(T_{i}) \subseteq \Th\biggl(T_{k+1} + \bigcup_{i \in X}\ThS(T_{i})\biggr).
\]
If $X \neq \mathrm{I}_{k}$, then $\bigcap_{i \in \mathrm{I}_{k} \setminus X}\ThP(T_{i}) = \ThP(\PA)$ by (i), so the inclusion holds. When $X = \mathrm{I}_{k}$, $T_{k+1} + \bigcup_{i \in X}\ThS(T_{i})$ is an inconsistent theory by (ii), so the inclusion must hold as well. 

2: Let $i_{0}, i_{1} \leq k+1$ be distinct natural numbers. 
Then $\bigcup_{\substack{j \neq i_{0} \\ j \leq k+1}}T_{j}$ is consistent by (iii) and therefore, $T_{i_{1}} + \bigcup_{\substack{j \neq i_{0}, i_{1} \\ j \leq k+1}}\ThS(T_{j})$ is consistent. 
By Theorem \ref{GHC2}, we have 
\[
\Biggl(\bigcap_{\substack{j \neq i_{0}, i_{1} \\ j \leq k+1}}\HCons(\Sigma_{n}, T_{j})\Biggr) \setminus \Th(T_{i_{1}}) \neq \emptyset.\qedhere
\]
\end{proof}

We obtain the following corollary by Theorems \ref{CE1} and~\ref{GC1}. 

\begin{cor}\label{CE1cor}
For any $k \geq 1$, there are theories $T_{0},\ldots,T_{k}$ and\/ $U$ satisfying the following conditions:
\begin{enumerate}
	\item $\bigcap_{i \leq k}\Cons(\Sigma_{n}, T_{i})\setminus \Th(U) = \emptyset$; 
	\item For all $i \leq k$, $\bigcap_{\substack{j \neq i \\ j \leq k}}\Cons(\Sigma_{n}, T_{j})\setminus\Th(U) \neq \emptyset$.\qed
\end{enumerate}
\end{cor}

Recall that Corollary~\ref{fCP} ruled out the existence of theories
satisfying the $\Cons(\Pi_{n},\cdot)$-analogue of condition~1 in Theorem~\ref{CE1}.

We take another look at the following conditions on infinite families introduced in Subsection~\ref{GCs}. 

\begin{description}
	\item[G1] $\bigcap_{i \in \omega}\Cons(\Gamma, T_{i}) \setminus \Th(T_{i})\neq \emptyset$. 
	\item[G2] For all $i \in \omega$, $\bigl(\bigcap_{\substack{j \neq i \\ j \in \omega}}\Cons(\Gamma, T_{j})\bigr) \setminus \Th(T_{i}) \neq \emptyset$. 
	\item[G3] For all $k \in \omega$, $\bigcap_{i \leq k}\Cons(\Gamma, T_{i}) \setminus \Th(T_{i})\neq \emptyset$. 
\end{description}

We are going to present a counterexample to the implication \textbf{G3}~$\Rightarrow$~\textbf{G2}. 
The following lemma will prepare us for the construction. 

\begin{lem}\label{L1}
Let $T$ be any theory which is not $\Sigma_1$-sound. 
Then there exists a $\Gamma$ sentence $\psi$ satisfying the following conditions: 
\begin{enumerate}
	\item $\psi$ is not provably equivalent to any $\Gamma^d$ sentence in $T$; 
	\item $\psi$ is not $\Gamma^d$-conservative over $T$. 
\end{enumerate}
\end{lem}
\begin{proof}
Let $\theta$ be a conventional $\Pi_1$~Rosser sentence for~$T$. 
Then $\theta$ is independent from~$T$. 

For $\Gamma = \Sigma_n$, let $\xi: \equiv \neg \theta$. 
Then $\xi$ is not $\Pi_1$-conservative over $T$ (See Lindstr\"om \cite[Exercise 5.1]{Lin}). 

For $\Gamma = \Pi_n$, let $\xi : \equiv \theta$. 
Since $T$ is not $\Sigma_1$-sound, $\xi$ is not $\Sigma_1$-conservative over $T$ (See Lindstr\"om \cite[Exercise 5.2.(b)]{Lin}). 

Let $\gamma$ be a $\Gamma$ sentence which is not provably equivalent to any $\Gamma^d$ sentence in $T + \xi$ (See Lindstr\"om \cite[Corollary 2.6]{Lin}). 
Then $\psi : \equiv \xi \land \gamma$ is a $\Gamma$ sentence satisfying the required conditions. 
\end{proof}

\begin{thm}\label{CCC}
There exists an infinite r.e.~family $\{T_i\}_{i \in \omega}$ of theories such that
\begin{enumerate}
	\item For all $k \in \omega$, $\bigcap_{i \leq k} \Cons(\Gamma, T_i) \setminus \Th(T_i)\neq \emptyset$; 
	\item $\bigl(\bigcap_{\substack{i \neq 0 \\ i \in \omega}} \Cons(\Gamma, T_i) \bigr) \setminus \Th(T_0) = \emptyset$. 
\end{enumerate}
\end{thm}

\begin{proof}
Let $T$ be a theory which is not $\Sigma_1$-sound. 
Let $(\varphi_i)_{i \geq 1}$ be any effective listing of all $\Gamma^d$ sentences with $T \vdash \varphi_1$. 
By Lemma \ref{L1}, there exists a $\Gamma$ sentence $\psi$ such that $\psi$ is not $T$-provably equivalent to any $\Gamma^d$ sentence and $\psi$ is not $\Gamma^d$-conservative over $T$. 
Then there exists a $\Gamma^d$ sentence $\xi$ such that $T + \psi \vdash \xi$ and $T \nvdash \xi$. 
Also $\psi$ is independent from~$T$ because $\psi$ is not $T$-equivalent to  $0=0$ nor to~$0 \neq 0$. 

Let $T_0 : = T + \neg \psi$ and for $i \geq 1$, $T_i : = T + \neg \varphi_i \lor \psi$. 
Since $T \nvdash \psi$ and $T \nvdash \neg \psi$, these theories are consistent. 
We prove that the family $\{T_i\}_{i \in \omega}$ satisfies the two conditions stated in the theorem. 
For this purpose, we prepare an increasing sequence $(X_k)_{k \geq 1}$ of finite sets of natural numbers in which each $X_k$ is a witness for condition \textbf{B1} for theories $T_1, \ldots, T_k$ and $T_0$. 
Let $\mathrm{D}_k : = \{1, 2, \ldots, k\}$. 
The increasing sequence $(X_k)_{k \geq 1}$ is inductively defined so that it satisfies the following three conditions for any $k \geq 1$:
\begin{enumerate}
	\item [(i)] $X_ k \subseteq \mathrm{D}_k$; 
	\item [(ii)] $T \nvdash \xi \lor \bigvee_{j \in X_k} \varphi_j$;  
	\item [(iii)] $T \vdash \bigvee_{j \in \mathrm{D}_k \setminus X_k} \neg \varphi_j \to \xi \lor \bigvee_{j \in X_k} \varphi_j$.
\end{enumerate}

Let $X_1 := \emptyset$. 
Then $\bigvee_{j \in X_1} \varphi_j \equiv \bot$. 
Since $T \nvdash \xi$, we have $T \nvdash \xi \lor \bigvee_{j \in X_1} \varphi_j$. 
Since $T \vdash \varphi_1$, we also have $T \vdash \bigvee_{j \in \mathrm{D}_1 \setminus X_1} \neg \varphi_j \to \xi \lor \bigvee_{j \in X_1} \varphi_j$. 

Suppose $X_k$ is already defined. 
We distinguish two cases. 
\begin{itemize}
	\item Case 1: $T \vdash \varphi_{k+1} \lor \xi \lor \bigvee_{j \in X_k} \varphi_j$. 

	Let $X_{k+1} : = X_k$. 

	Since $T \nvdash \xi \lor \bigvee_{j \in X_k} \varphi_j$, we obtain $T \nvdash \xi \lor \bigvee_{j \in X_{k+1}} \varphi_j$. 
	From condition~(iii) for $X_k$, $T \vdash \bigvee_{j \in \mathrm{D}_k \setminus X_k} \neg \varphi_j \to \xi \lor \bigvee_{j \in X_k} \varphi_j$. 
	By the assumption of Case~1, $T \vdash \neg \varphi_{k+1} \to \xi \lor \bigvee_{j \in X_k} \varphi_j$. 
	Since $j \in \mathrm{D}_{k+1} \setminus X_{k+1}$ if and only if $j \in \mathrm{D}_k \setminus X_k$ or $j = k+1$, we obtain $T \vdash \bigvee_{j \in \mathrm{D}_{k+1} \setminus X_{k+1}} \neg \varphi_j \to \xi \lor \bigvee_{j \in X_{k}} \varphi_j$. 
	Therefore, $T \vdash \bigvee_{j \in \mathrm{D}_{k+1} \setminus X_{k+1}} \neg \varphi_j \to \xi \lor \bigvee_{j \in X_{k+1}} \varphi_j$. 

	\item Case 2: $T \nvdash \varphi_{k+1} \lor \xi \lor \bigvee_{j \in X_k} \varphi_j$. 

	Let $X_{k+1} : = X_k \cup \{k+1\}$. 

	Then $T \nvdash \xi \lor \bigvee_{j \in X_{k+1}} \varphi_j$. 
	From (iii) for $X_k$, $T \vdash \bigvee_{j \in \mathrm{D}_k \setminus X_k} \neg \varphi_j \to \xi \lor \bigvee_{j \in X_k} \varphi_j$. 
	Since 	$\mathrm{D}_{k+1} \setminus X_{k+1} = \mathrm{D}_k \setminus X_k$ and $X_{k+1}\supseteq X_{k}$, we obtain $T \vdash \bigvee_{j \in \mathrm{D}_{k+1} \setminus X_{k+1}} \neg \varphi_j \to \xi \lor \bigvee_{j \in X_{k+1}} \varphi_j$. 
\end{itemize}

The definition is completed. 
We shall prove clauses 1 and 2 of the Theorem.   

1. Fix any $k$. 
For each $i \leq k$, we show $\bigl(\bigcap_{\substack{j \neq i \\ j \leq k}} \Cons(\Gamma, T_j) \bigr) \setminus \Th(T_i) \neq \emptyset$. 

For $i = 0$, let $\theta_k$ be the $\Gamma^d$ sentence $\xi \lor \bigvee_{j \in X_k} \varphi_j$. 
Then we have $T \vdash \bigvee_{j \in \mathrm{D}_k \setminus X_k} \neg \varphi_j \to \theta_k$. 
Since $T + \psi \vdash \xi$, we obtain $T + \psi \vdash \theta_k$. 
Therefore $T_j \vdash \theta_k$ for all $j \in \mathrm{D}_k \setminus X_k$. 
This means $\theta_k \in \bigcap_{j \in \mathrm{D}_k \setminus X_k} \Th_{\Gamma^d}(T_j)$. 

Suppose $T_0 + \bigcup_{j \in X_k} \ThG(T_j) \vdash \theta_k$. 
That is, 
\[
	T + \bigwedge_{j \in X_k} (\neg \varphi_j \lor \psi) + \neg \psi \vdash \xi \lor \bigvee_{j \in X_k} \varphi_j. 
\]
Then 
\[
	T + \bigwedge_{j \in X_k} \neg \varphi_j + \neg \psi \vdash \xi \lor \bigvee_{j \in X_k} \varphi_j. 
\]
Since $T + \neg \bigwedge_{j \in X_k} \neg \varphi_j \vdash \bigvee_{j \in X_k} \varphi_j$ and $T + \psi \vdash \xi$, we obtain $T \vdash \xi \lor \bigvee_{j \in X_k} \varphi_j$. 
But this contradicts condition~(ii) for~$X_{k}$, 
so $\theta_k \notin \Th\bigl(T_0 + \bigcup_{j \in X_k} \ThG(T_j)\bigr)$. 

Thus $\theta_k$ witnesses the non-inclusion 
\[
\bigcap_{j \in \mathrm{D}_k \setminus X_k} \Th_{\Gamma^d}(T_j) \nsubseteq \Th\biggl(T_0 + \bigcup_{j \in X_k}\ThG(T_j)\biggr).
\] 
By Theorem \ref{OSC}, we conclude
\begin{align}\label{f1}
	\Biggl(\bigcap_{\substack{j \neq 0 \\ j \leq k}} \Cons(\Gamma, T_j) \Biggr) \setminus \Th(T_0) \neq \emptyset. 
\end{align}

For $i \neq 0$, suppose that the theory $T_i + \bigcup_{\substack{j \neq i \\ j \leq k}} \ThG(T_j)$ is inconsistent. 
Then $\bigcup_{\substack{j \neq 0 \\ j \leq k}}T_{j} + \ThG(T_{0})$ is inconsistent. 
Notice that for each $j \geq 1$, the theory $T_j = T + \neg \varphi_j \lor \psi$ is a subtheory of $T + \psi$. 
Hence $T + \ThG(T + \neg \psi) + \psi$ is inconsistent. 
Then there exists a $\Gamma$ sentence $\gamma$ such that $T + \neg \psi \vdash \gamma$ and $T + \gamma + \psi$ is inconsistent. 
Then we obtain $T \vdash \psi \leftrightarrow \neg \gamma$, but with $\neg\gamma$ being~$\Gamma^{d}$, this contradicts our choice of~$\psi$. 
Therefore $T_{i} + \bigcup_{\substack{j \neq i \\ j \leq k}} \ThG(T_j)$ is consistent. 
We obtain $\bigl(\bigcap_{\substack{j \neq i \\ j \leq k}} \Cons(\Gamma, T_j) \bigr) \setminus \Th(T_i) \neq \emptyset$ by Theorem \ref{OSC}. 
By combining this with (\ref{f1}), we conclude
\[
	\text{for any}\ i \leq k,\ \Biggl(\bigcap_{\substack{j \neq i \\ j \leq k}} \Cons(\Gamma, T_j) \Biggr) \setminus \Th(T_i) \neq \emptyset. 
\]

By Theorem \ref{GC1}, this is equivalent to $\bigcap_{i \leq k}\Cons(\Gamma, T_i) \setminus \Th(T_i)\neq \emptyset$. 

2. It suffices to prove that for any $i \geq 1$, $\varphi_i \notin \Cons(\Gamma, T_i)$ or $T_0 \vdash \varphi_i$. 
Clearly, $T_i + \varphi_i \vdash \psi$. 
If $T_i \nvdash \psi$, then $\varphi_i \notin \Cons(\Gamma, T_i)$ because $\psi$ is a $\Gamma$ sentence. 
If $T_i \vdash \psi$, then $T + \neg \varphi_i \lor \psi \vdash \psi$. 
Thus $T \vdash \neg \varphi_i \to \psi$. 
Hence $T_0 \vdash \varphi_i$. 
\end{proof}

Since $\bigcap_{\substack{i \in \omega}}\Cons(\Gamma, T_i) \setminus \Th(T_i)= \emptyset$ obviously follows from the second clause of Theorem~\ref{CCC}, we obtain the following corollary. 
This is a counterexample to the implication \textbf{G3}~$\Rightarrow$~\textbf{G1}.

\begin{cor}\label{CfC}
There exists an infinite r.e.~family $\{T_i\}_{i \in \omega}$ of theories satisfying the following two conditions:
\begin{enumerate}
	\item For all $k \in \omega$, $\bigcap_{i \leq k}\Cons(\Gamma, T_i) \setminus \Th(T_i)\neq \emptyset$; 
	\item $\bigcap_{\substack{i \in \omega}}\Cons(\Gamma, T_i) \setminus \Th(T_i)= \emptyset$.\qed
\end{enumerate}
\end{cor}

We return to conditions introduced in Subsection \ref{GCs}. 
Let us focus on infinite families $\{T_i\}_{i \in \omega}$.
\begin{description}
	\item [B1] There exists an r.e.~set $X \subseteq \omega$ such that 
\[
\bigcap_{i \in \omega \setminus X} \Th_{\Gamma^d}(T_i) \nsubseteq \Th\biggl(U + \bigcup_{i \in X} \ThG(T_i)\biggr).
\] 
	\item[B2] There exists a set $X \subseteq \omega$ such that 
\[
\bigcap_{i \in \omega \setminus X} \Th_{\Gamma^d}(T_i) \nsubseteq \Th\biggl(U + \bigcup_{i \in X} \ThG(T_i)\biggr).
\] 
	\item[B3] $\bigcap_{i \in \omega}\Cons(\Gamma, T_{i})\setminus \Th(U) \neq \emptyset$. 
	\item[B4] $\bigcap_{i \leq k}\Cons(\Gamma,T_{i})\setminus\Th(U)\neq \emptyset$
	   for all $k \in \omega$. 
\end{description}

From Theorem~\ref{CCC}, we obtain a counterexample to the implication \textbf{B4}~$\Rightarrow$~\textbf{B3}. 

\begin{cor}\label{counterexample1}
There exist an infinite r.e.~family $\{T_i\}_{i \in \omega}$ of theories and a theory\/~$U$ satisfying the following two conditions:
\begin{enumerate}
	\item For all $k \in \omega$, $\bigcap_{\substack{i \leq k}} \Cons(\Gamma, T_i) \setminus \Th(U) \neq \emptyset$; 
	\item $\bigcap_{\substack{i \in \omega}} \Cons(\Gamma, T_i)\setminus \Th(U) = \emptyset$.\qed
\end{enumerate}
\end{cor}

Finally, we construct a counterexample to the implication \textbf{B2} $\Rightarrow$ \textbf{B1}.

\begin{thm}\label{SCex}
For any $\Pi_{1}$ set $X \subseteq \omega$, there exist an infinite r.e.~family $\{T_i\}_{i \in \omega}$ and a theory\/ $U$ such that for any\/ $Y \subseteq \omega$, 
\[
\bigcap_{i \in \omega \setminus Y} \Th_{\Gamma^d}(T_i) \nsubseteq \Th\biggl(U + \bigcup_{i \in Y} \ThG(T_i)\biggr) \ \text{if and only if\/} \ Y = X.
\] 
\end{thm}

\begin{proof}
Let $X$ be any $\Pi_{1}$ set. 
Let $T$ be some theory which is not $\Sigma_1$-sound. 
By Lemma \ref{L1}, there exist a $\Gamma$ sentence $\psi$ and a $\Gamma^d$ sentence $\xi$ such that $T \nvdash \psi$, $T \nvdash \neg \psi$, $T + \psi \vdash \xi$ and $T \nvdash \xi$. 
Since $\omega \setminus X$ is an r.e.~set, by Fact \ref{F1}, there exists a $\Gamma^d$ formula $\delta(x)$ satisfying the following conditions for any $i \in \omega$: 
\begin{itemize}
	\item If $i \notin X$, then $T + \neg \psi \vdash \delta(\overline{i})$; 
	\item If $i \in X$, then $\neg \delta(\overline{i})$ is $\Gamma^d$-conservative over $T + \neg \psi$. 
\end{itemize}
Let $T_i : = T + \neg \delta(\overline{i}) \lor \psi$ for $i \in \omega$ and $U : = T + \neg \psi$. 

\begin{cl}\label{Cl2} The family $\{T_{i}\}_{i \in \omega}$ satisfies the following conditions for any $i \in \omega$: 
\begin{itemize}
	\item[\upshape(i)] If $i \notin X$, then $T_i$ is deductively equivalent to $T + \psi$; 
	\item[\upshape(ii)] If $i \in X$, then $\ThD(T_i) \subseteq \Th(T)$. 
\end{itemize}
\end{cl}
\begin{proof}
(i). Suppose $i \notin X$. 
Then $T + \neg \psi \vdash \delta(\overline{i})$. 
Thus $T \vdash (\neg \delta(\overline{i}) \lor \psi) \leftrightarrow \psi$. 
This means that $T_i = T + \neg \delta(\overline{i}) \lor \psi$ is deductively equivalent to $T + \psi$. 

(ii). Suppose $i \in X$. 
Then $\neg \delta(\overline{i})$ is $\Gamma^d$-conservative over $T + \neg \psi$. 
For an arbitrary $\Gamma^d$ sentence $\varphi$, suppose $T_i \vdash \varphi$. 
Then $T + \neg \delta(\overline{i}) \lor \psi \vdash \varphi$, and hence $T + \neg \delta(\overline{i}) \vdash \varphi$ and $T + \psi \vdash \varphi$. 
We have $T + \neg \psi + \neg \delta(\overline{i}) \vdash \varphi$. 
By the $\Gamma^{d}$-conservativity of $\lnot \delta(\overline{i})$, $T + \neg \psi \vdash \varphi$. 
Hence $T \vdash \varphi$. 
\end{proof}

We return to the proof of the theorem.
First, we show $\bigcap_{i \in \omega \setminus X} \Th_{\Gamma^d}(T_i) \nsubseteq \Th\bigl(U + \bigcup_{i \in X} \ThG(T_i)\bigr)$. 

By clause~(i) of the Claim, for any $i \notin X$, $T_i \vdash \xi$ because $T + \psi \vdash \xi$. 
Then $\xi \in \bigcap_{i \in \omega \setminus X} \Th_{\Gamma^d}(T_i)$. 

Suppose, towards  contradiction, that the theory $U + \bigcup_{i \in X}\ThG(T_i)$ proves~$\xi$. 
Then there are $i_0, \ldots, i_{k-1} \in X$ such that $T + \bigwedge_{l < k} (\neg \delta(\overline{i_l}) \lor \psi) + \neg \psi \vdash \xi$. 
Thus $T + \bigwedge_{l < k} \neg \delta(\overline{i_l}) + \neg \psi \vdash \xi$.
By Lemma \ref{cCons}.1, we obtain $T + \neg \psi \vdash \xi$. 
Since $T + \psi \vdash \xi$, it follows that $T \vdash \xi$, contradicting the choice of~$\xi$.
Therefore $U + \bigcup_{i \in X}\ThG(T_i) \nvdash \xi$. 
We conclude $\bigcap_{i \in \omega \setminus X} \Th_{\Gamma^d}(T_i) \nsubseteq \Th(U + \bigcup_{i \in X} \ThG(T_i))$. 

Next, we prove that if $Y \neq X$, then 
\[
\bigcap_{i \in \omega \setminus Y} \Th_{\Gamma^d}(T_i) \subseteq \Th\biggl(U + \bigcup_{i \in Y} \ThG(T_i)\biggr).
\] 
Let $Y \subseteq \omega$ be such that $Y \neq X$. 
We distinguish the following two cases. 
\begin{itemize}
	\item Case 1: $Y \nsubseteq X$. 
	
Let $j\in Y\setminus X$. 
Then by clause~(i) of the Claim, $T_j$ is deductively equivalent to $T + \psi$. 
Thus $\psi \in \ThG(T_j)$. 
Since $U = T + \neg \psi$,
we conclude that $U + \bigcup_{i \in Y} \ThG(T_i)$ is inconsistent. 
Therefore the inclusion 
$\bigcap_{i \in \omega \setminus Y} \Th_{\Gamma^d}(T_i) \subseteq \Th\bigl(U + \bigcup_{i \in Y} \ThG(T_i)\bigr)$ holds trivially. 

	\item Case 2: $X \nsubseteq Y$. 
	
Fix $j\in X\setminus Y$.
Let $\varphi$ be any $\Gamma^d$ sentence with $\varphi\in\bigcap_{i\in\omega \setminus Y} \Th_{\Gamma^d}(T_i)$. 
Then $T_j \vdash \varphi$. 
By clause~(ii) of the Claim, $\ThD(T_j) \subseteq \Th(T)$. 
Thus $T \vdash \varphi$. 
Since $U$ is an extension of $T$, $U$ also proves $\varphi$. 
This shows $\bigcap_{i \in \omega \setminus Y} \Th_{\Gamma^d}(T_i) \subseteq \Th\bigl(U + \bigcup_{i \in Y} \ThG(T_i)\bigr)$. 
\end{itemize}

Therefore, $\{T_i\}_{i \in \omega}$ and $U$ satisfy the required conditions. 
\end{proof}

Theorem \ref{SCex} leads to a counterexample to the implication \textbf{B2} $\Rightarrow$ \textbf{B1}. 

\begin{cor} \label{MSCex}
There exist an infinite r.e.~family $\{T_i\}_{i \in \omega}$ of theories and a theory\/ $U$ satisfying the following two conditions:
\begin{enumerate}
	\item There exists a set $X \subseteq \omega$ such that 
\[
\bigcap_{i \in \omega \setminus X} \Th_{\Gamma^d}(T_i) \nsubseteq \Th\biggl(U + \bigcup_{i \in X} \ThG(T_i)\biggr);
\] 
	\item There is no r.e.~set $X \subseteq \omega$ such that 
\[
\bigcap_{i \in \omega \setminus X} \Th_{\Gamma^d}(T_i) \nsubseteq \Th\biggl(U + \bigcup_{i \in X} \ThG(T_i)\biggr).
\] 
\end{enumerate}
\end{cor}

\begin{proof}
Let $X \subseteq \omega$ be a $\Pi_{1}$ set which is not r.e. 
Let $\{T_{i}\}_{i \in \omega}$ be an r.e.~family of theories and $U$ a theory corresponding to that~$X$ as in Theorem \ref{SCex}. 
Then, 
\[
\bigcap_{i \in \omega \setminus X} \Th_{\Gamma^d}(T_i) \nsubseteq \Th\biggl(U + \bigcup_{i \in X} \ThG(T_i)\biggr).
\] 
Furthermore, for any r.e.~set $Y$, 
\[
\bigcap_{i \in \omega \setminus Y} \Th_{\Gamma^d}(T_i) \subseteq \Th\biggl(U + \bigcup_{i \in Y} \ThG(T_i)\biggr)
\]
because $Y \neq X$. 
\end{proof}

\bibliographystyle{plain}
\bibliography{ref}

\end{document}